\documentclass[reqno,twoside,11pt,english]{amsart}

\usepackage[T1]{fontenc}
\usepackage{amsmath,amssymb,amsthm,mathtools}
\usepackage{mathrsfs}
\usepackage{graphicx,float}
\usepackage{booktabs,tabularx}
\usepackage{comment}
\usepackage{color}
\usepackage[colorlinks=true,linkcolor=blue,citecolor=blue,urlcolor=blue]{hyperref}

\allowdisplaybreaks

\numberwithin{equation}{section}

\newtheorem{theorem}{Theorem}[section]
\newtheorem{proposition}[theorem]{Proposition}
\newtheorem{lemma}[theorem]{Lemma}

\newtheorem{remark}[theorem]{Remark}

\newcommand{\R}{\mathbb{R}}
\newcommand{\C}{\mathbb{C}}

\newcommand{\eps}{\varepsilon}
\newcommand{\cX}{\mathcal{X}}
\newcommand{\cN}{\mathcal{N}}
\newcommand{\cZ}{\mathcal{Z}}

\begin{document}

\title[General-coefficient singular limits for energy-critical CGL equations]
{Singular limits for general energy-critical complex Ginzburg--Landau equations}

\author{Lingbang Gao}
\address{School of Mathematics, Shandong University,
Jinan 250100, Shandong, P. R. China}
\email{lingbang\_gao@mail.sdu.edu.cn}

\author{Jie Xin}
\address{School of Information Engineering, Shandong Youth University
of Political Science, Jinan 250103, P. R. China}
\email{fdxinjie@sina.com}

\author{Yunrui Zheng}
\address{School of Mathematics, Shandong University,
Jinan 250100, Shandong, P. R. China}
\email{yunrui\_zheng@sdu.edu.cn}

\begin{abstract}
Motivated by dissipative approximations of the three-dimensional focusing
energy-critical nonlinear Schr\"odinger equation, we study the inviscid and
zero-dispersion limits for general energy-critical complex
Ginzburg--Landau equations in dimensions \(3\le d\le6\).  We develop a new
framework for these limits that does not require a global theory for the
limiting equation.
For limiting data
in \(H^1\), we prove strong convergence in the critical space, in
particular in \(L_t^\infty H_x^1\), on every compact subinterval of the
maximal lifespan of the limiting solution.  This includes strong
\(H^1\) inviscid convergence to the 3D focusing energy-critical NLS without
higher regularity of the limiting data.  The diffusion and nonlinear
coefficients may vary independently, and both focusing and defocusing cases
are treated.  Both limits use coefficient-uniform critical stability.  The
inviscid argument uses low-frequency truncation to overcome the
derivative loss at \(H^1\), whereas uniform parabolic smoothing yields a
quantitative \(H^1\) estimate in the zero-dispersion limit. Moreover, for the inviscid limit, we prove the required homogeneous Strichartz estimates and obtain the retarded estimates by establishing a uniform retarded double-endpoint bound for the critical forcing space.

\medskip
\noindent\textbf{Keywords}: Complex Ginzburg--Landau equations,
energy-critical nonlinear Schr\"odinger equation, inviscid limit,
zero-dispersion limit.

\smallskip
\noindent\textbf{Mathematics Subject Classification (2020).} Primary 35Q56;
Secondary 35B25, 35Q55.
\end{abstract}

\maketitle

\section{Introduction}\label{sec:introduction}

We study the inviscid and zero-dispersion limits for energy-critical complex
Ginzburg--Landau equations.  Previous energy-critical singular-limit
arguments often rely on global theory, while strong \(H^1\) convergence
has required higher regularity of the initial data for the limiting equation.
This raises the question: can one prove strong convergence in
\(L_t^\infty H_x^1\) from \(H^1\) initial data without relying on global
well-posedness or global spacetime bounds for the limiting equation?
We prove this convergence on compact intervals within the limiting
solution's lifespan using local critical stability and coefficient-uniform
estimates, without requiring a global theory for the limiting equation.
The argument permits independent variation of the diffusion and nonlinear
coefficients and treats both focusing and defocusing cases.

We consider the energy-critical complex Ginzburg--Landau (CGL) equation 
\begin{equation}
\label{eq:intro-cgl}
 \partial_tu-(a+i\alpha)\Delta u
 +\mu(b+i\beta)f(u)+Ru=0,
 \qquad (t,x)\in[0,\infty)\times\R^d,
\end{equation}
where
\[
 d\in\{3,4,5,6\},\qquad
 f(u)=|u|^pu,\qquad p=\frac4{d-2},
\]
\[
 a,b>0,\qquad \alpha,\beta\in\R,\qquad R\ge0,
 \qquad \mu\in\{-1,1\}.
\]
Here \(\mu=1\) is the defocusing case, \(\mu=-1\) is the focusing
case, and \(Ru\) is a linear damping term.  When \(R=0\), the equation
is invariant under the energy-critical scaling
\(
 u(t,x)\mapsto
 \lambda^{\frac{d-2}{2}}u(\lambda^2t,\lambda x),
\)
which preserves \(\dot H^1(\R^d)\). For the origins and applications of Ginzburg--Landau theory, see
\cite{GinzburgLandau1950,AransonKramer2002,GuoJiangLi2020}.  The CGL Cauchy
theory is treated in
\cite{GinibreVelo1996,GinibreVelo1997,ClementOkazawaSobajimaYokota2012}.
Related qualitative results include
\cite{OkazawaYokota2002Monotonicity,OkazawaYokota2002JDE,
CorreiaFigueira2018,ChengGuoZhengZheng2024Exterior}.

Write \(A=a+i\alpha\) and \(C=b+i\beta\).  When
\(A=C=-i\) and \(R=0\), \eqref{eq:intro-cgl} reduces to the
energy-critical nonlinear Schr\"odinger (NLS) equation (inviscid limit);
when \(A=C=1\) and \(R=0\), it reduces to the energy-critical nonlinear
heat equation (zero-dispersion limit).  For critical NLS local and
persistence theory, see \cite{CazenaveWeissler1990,Tao2006}; for dispersive
estimates relevant to the inviscid limit, see
\cite{Strichartz1977,GinibreVelo1992,KeelTao1998,ChristKiselev2001}
(see also~\cite{MiaoZhangZheng2025}).

An important motivation for the inviscid limit is the dissipative
approximation of the three-dimensional focusing energy-critical NLS
\begin{equation}\label{eq:NLS}
 i\partial_t v-\Delta v-|v|^4v=0.
\end{equation}
Cheng, Guo, and Zheng~\cite{ChengGuoZheng2026JMPA} used this approximation to
construct global weak solutions of the 3D focusing NLS for
$H^1$ data without radial symmetry.  Here we seek strong inviscid
convergence at the same regularity on compact subintervals of the
maximal lifespan of the limiting solution.

Previous energy-critical singular-limit arguments rely heavily on global
information for both the CGL approximations and the limiting NLS.  These
arguments do not yield strong inviscid convergence at
the natural \(H^1\) regularity for \eqref{eq:NLS}.
The diffusion and nonlinear coefficients are allowed to vary independently.
The phase of \(A\) determines the balance between dissipation and dispersion,
while the phase of \(C\) enters the nonlinearity independently.
When the phases of \(A\) and \(C\) are not aligned, the usual
energy calculation for the general CGL \eqref{eq:intro-cgl} contains a mixed term without a fixed sign, so the
approximating equations need not share a monotone energy or Lyapunov
functional.  The global dynamics for general CGL \eqref{eq:intro-cgl} are
therefore not fully understood, and most previous studies impose restrictions
on the coefficients.  In the energy-critical setting, earlier approaches
address these difficulties either by requiring higher regularity of the
limiting NLS data or, at the natural \(H^1\) regularity, by obtaining only
weaker forms of convergence.

For energy-subcritical inviscid limits, Wu~\cite{Wu1998} proves
\(L_t^\infty H_x^1\) convergence for \(H^2\) data in both focusing and
defocusing cases.  For defocusing case and \(H^1\) data,
Wang~\cite{Wang2002} obtains quantitative \(L_t^\infty L_x^2\) convergence,
while Machihara and Nakamura~\cite{MachiharaNakamura2003} prove strong
\(L_t^\infty H_x^1\) convergence.  At the defocusing
energy-critical exponent, Bechouche and J\"ungel
\cite{BechoucheJungel2000} obtain weak-\(*\) convergence in
\(L_t^\infty H_x^1\) from \(H^1\) data, together with strong convergence in
subcritical local Lebesgue spaces.  The strong \(H^1\)-level result of Huang
and Wang \cite{HuangWang2008} requires \(H^3\) initial data for NLS.  More
recently, Cheng, Guo, and Zheng \cite{ChengGuoZheng2025Limit} study both
limits for an aligned energy-critical family.  These results leave a
regularity gap at the energy-critical level between the natural \(H^1\) data
and the higher regularity required for strong \(H^1\) convergence.

The present paper removes this gap.  For limiting data in \(H^1\), we prove
strong convergence in the full critical space \(\cX^1\), defined in
Section~\ref{sec:prelim}, and hence in \(L_t^\infty H_x^1\).  In the
inviscid case, this includes the focusing equation on every compact
subinterval of the maximal lifespan of the limiting NLS solution.  See
Table~\ref{tab:cgl-limit-comparison} for a detailed comparison.

\begin{table}[H]
\centering
\caption{Comparison with representative singular-limit results for CGL.}
\label{tab:cgl-limit-comparison}
\footnotesize
\setlength{\tabcolsep}{4pt}
\renewcommand{\arraystretch}{1.13}
\resizebox{\textwidth}{!}{
\begin{tabular}{@{}lcccccc@{}}
\toprule
Reference & Focusing/Defocusing & Data & Convergence & Limit & Energy-critical
  & Coefficients \\
\midrule
JDE (1998)~\cite{Wu1998}
  & Both & $H^2$ & $L^{\infty}H^1$ & Inviscid & No & Conditional \\
CPAM (2002)~\cite{Wang2002}
  & Defocusing & $H^1$ & $L^{\infty}L^2$ & Inviscid & No & Conditional \\
JMAA (2003)~\cite{MachiharaNakamura2003}
  & Defocusing & $H^1$ &  $L^{\infty}H^1$ & Inviscid & No
  & Conditional \\
CMP (2000)~\cite{BechoucheJungel2000}
  & Defocusing & $H^1$ & $L^{\infty}H^1$ weak-$*$ & Inviscid & Yes & Conditional \\
JFA (2008)~\cite{HuangWang2008}
  & Defocusing & $H^3$ &  $L^{\infty}H^1$
  & Inviscid & Yes & Conditional \\
Acta Math. Sin. (2025)~\cite{ChengGuoZheng2025Limit}
  & Both & $H^2_{\rm heat}/H^3_{\rm NLS}$ & $L^{\infty}H^1$ & Both & Yes & Aligned \\
\textbf{Present work}
  & \textbf{Both} & $\boldsymbol{H^1}$ & $\boldsymbol{\cX^1} (\subset L^{\infty}H^1) $
  & \textbf{Both} & \textbf{Yes} & \textbf{Independent} \\
\bottomrule
\end{tabular}
}
\vspace{2pt}

\parbox{\textwidth}{\scriptsize Here ``both'' in the limit
column means the zero-dispersion and inviscid limits.  }
\end{table}

The results cover dimensions \(3\le d\le6\),
both focusing and defocusing cases, independently varying
diffusion and nonlinear coefficients, and both the inviscid and
zero-dispersion limits. 
The inviscid analysis also yields coefficient-uniform homogeneous and retarded Strichartz estimates, with the retarded estimates obtained from a uniform retarded double-endpoint bound for the critical forcing space.

\subsection{Main result}

We first normalize the
coefficients.  If the original coefficients
are \(\widetilde A_\eps,\widetilde C_\eps\neq0\), set
\(
 \lambda_\eps=|\widetilde A_\eps|,
 ~
 \rho_\eps=
 (\lambda_\eps / |\widetilde C_\eps|)^{1/p},
 ~
 u^\eps(t,x)=
 \rho_\eps^{-1}\widetilde u^\eps(t / \lambda_\eps,x).
\)
The normalized equation is
\begin{equation}\label{eq:intro-normalized-family}
 \partial_tu^\eps-A_\eps\Delta u^\eps
 +\mu C_\eps f(u^\eps)+r_\eps u^\eps=0,
 \qquad r_\eps\ge0,
\end{equation}
where
\(
 A_\eps=\widetilde A_\eps / |\widetilde A_\eps|,
 ~
 C_\eps=\widetilde C_\eps / |\widetilde C_\eps|,
 ~
 r_\eps=\widetilde R_\eps / |\widetilde A_\eps|.
\)
The inviscid limit corresponds to \(A_\eps,C_\eps\to-i\), whereas the
zero-dispersion limit corresponds to \(A_\eps,C_\eps\to1\); in both cases
\(r_\eps\to0\).

For \(\mu\in\{-1,1\}\), let \(v\) denote the maximal \(H^1\) mild solution of
\begin{equation}
 i\partial_tv-\Delta v+\mu f(v)=0,
 \qquad v(0)=v_0,
 \label{eq:intro-nls}
\end{equation}
with maximal lifespan \(T_{\max}(v_0)\).  Here and below, a maximal \(H^1\)
mild solution is understood in the standard critical solution space: on
every compact subinterval \(I\) of its maximal lifespan, it belongs to
\(
 C(I;H^1(\R^d))\cap\cX^1(I).
\)
The spaces \(\cX^1(I)\) and \(\cN^1(I)\) are defined in Section~\ref{sec:prelim}.

\begin{theorem}[Inviscid limit]
\label{thm:intro-B}
Let \(d\in\{3,4,5,6\}\), and fix \(\mu\in\{-1,1\}\).
Let \(v_0\in H^1(\R^d)\), and let \(v\) be the maximal \(H^1\) mild
 solution of \eqref{eq:intro-nls}.  Assume that
 \(
 A_\eps\to-i,~
 C_\eps\to-i,~
 r_\eps\to0^+,~
 \|u_0^\eps-v_0\|_{H^1}\to0.
\)
Then, for every \(0<T<T_{\max}(v_0)\), equation
\eqref{eq:intro-normalized-family}, with initial data \(u^\eps(0)=u_0^\eps\), admits, for all
sufficiently small \(\eps\), a unique solution
\[
 u^\eps\in C([0,T];H^1(\R^d))\cap\cX^1([0,T]),
\]
and
\begin{equation}\label{eq:intro-schrodinger-limit-H1}
 u^\eps\to v
 \qquad\text{in }\cX^1([0,T])
 \quad\text{as }\eps\to0.
\end{equation}

If, in addition, \(v_0\in H^{1+s}(\R^d)\) for some fixed
\(0<s\le2\), then there exists \(K_{v,T,s}>0\), independent of
\(\eps\), such that
\begin{equation}\label{eq:intro-schrodinger-limit}
 \|u^\eps-v\|_{\cX^1([0,T])}
 \le K_{v,T,s}\bigl(
 \|u_0^\eps-v_0\|_{H^1}
 +|A_\eps+i|^{s/2}
 +|C_\eps+i|
 +r_\eps
 \bigr)
\end{equation}
for all sufficiently small \(\eps\).
\end{theorem}

The same framework also applies to the zero-dispersion limit.  For
\(\mu\in\{-1,1\}\), let \(h\) denote the maximal \(H^1\) mild solution of
\begin{equation}
 \partial_th-\Delta h+\mu f(h)=0,
 \qquad h(0)=h_0,
 \label{eq:intro-heat}
\end{equation}
with maximal lifespan \(T_{\max}(h_0)\).
The critical parabolic theory relevant to this limit dates back to Weissler
and Giga~\cite{Weissler1980,Giga1986}.  Related references include
\cite{MiaoYuanZhang2008,Henry1981,Pazy1983,Lunardi1995}.

\begin{theorem}[Zero-dispersion limit]
\label{thm:intro-zero-dispersion}
Let \(d\in\{3,4,5,6\}\), and fix \(\mu\in\{-1,1\}\).
Let \(h_0\in H^1(\R^d)\), and let \(h\) be the maximal \(H^1\) mild
solution of \eqref{eq:intro-heat}.  Assume that
\(
 A_\eps\to1,~
 C_\eps\to1,~
 r_\eps\to0^+,~
 \|u_0^\eps-h_0\|_{H^1}\to0.
\)
Then, for every \(0<T<T_{\max}(h_0)\), equation
\eqref{eq:intro-normalized-family} with initial data
\(u^\eps(0)=u_0^\eps\) admits, for all sufficiently small \(\eps\), a
unique solution
\[
 u^\eps\in C([0,T];H^1(\R^d))\cap\cX^1([0,T]),
\]
and
\begin{equation}\label{eq:intro-heat-limit-H1}
 u^\eps\to h
 \qquad\text{in }\cX^1([0,T])
 \quad\text{as }\eps\to0.
\end{equation}

Moreover, there exists \(K_{h,T}>0\), independent of \(\eps\), such that
\begin{equation}\label{eq:intro-heat-limit}
 \|u^\eps-h\|_{\cX^1([0,T])}
 \le K_{h,T}\bigl(
 \|u_0^\eps-h_0\|_{H^1}
 +|A_\eps-1|
 +|C_\eps-1|
 +r_\eps
 \bigr)
\end{equation}
for all sufficiently small \(\eps\).
\end{theorem}

\begin{remark}
For the inviscid limit, general \(H^1\) data require low-frequency
truncation to avoid derivative loss in \((A_\eps+i)\Delta v\).
The truncation errors have no universal algebraic rate at this regularity,
while the coefficient error grows with the cutoff frequency; hence one
first fixes the cutoff and then lets the coefficients converge.
This lack of a rate is intrinsic, rather than merely a technical consequence
of the low-frequency truncation.  In fact, the linear obstruction in
Section~\ref{subsec:linear-rate-obstruction}
shows that \(H^1\) regularity alone cannot guarantee an algebraic inviscid
rate, even at the linear level; see also
Remark~\ref{rem:sch-two-parameter}.  For \(H^{1+s}\) data,
\(0<s\le2\), we use NLS solutions with low-frequency truncated
initial data.  Balancing \(N^{-s}\) against
\(|A_\eps+i|N^{2-s}\) yields the rate \(|A_\eps+i|^{s/2}\) in
Theorem~\ref{thm:intro-B}, recovering the linear coefficient rate at \(s=2\).
By contrast, uniform parabolic smoothing controls the diffusion-coefficient
error directly in the zero-dispersion limit, so the quantitative estimate
in Theorem~\ref{thm:intro-zero-dispersion} already holds for \(H^1\) data.
\end{remark}

\subsection{Main ideas and analytic difficulties}

Both singular limits use the common coefficient-uniform critical stability
argument, but the diffusion-coefficient errors \((A_\eps+i)\Delta v\) and
\((A_\eps-1)\Delta h\) are treated differently.
The inviscid limit presents a further, essential difficulty: the linear
estimates must remain uniform as the dissipative part tends to zero.
Establishing this uniformity is the main challenge in the linear analysis.

\smallskip
\noindent\emph{Stability with general coefficients.}
The proof is carried out on compact subintervals of the maximal lifespan of
the limiting solution.  Since independent complex coefficients need not
provide a common monotone energy or Lyapunov functional, we use a local
critical stability argument in the spaces \(\cX^1\) and \(\cN^1\),
introduced in Section~\ref{sec:prelim}.  Each compact interval is divided
into finitely many pieces on which the critical norm is small, and the
stability constants are uniform in the coefficients.  Thus the argument
uses only the critical spacetime control of the limiting solution on the
interval under consideration; it requires neither global well-posedness nor
global spacetime bounds for the limiting equation.  See
Proposition~\ref{prop:common-stability}.
The range \(3\le d\le6\) ensures
\(p=4/(d-2)\ge1\), so \(Df\) is locally Lipschitz and the same critical
nonlinear difference estimate applies at the zero-dispersion and inviscid
limits.

\smallskip
\noindent\emph{Derivative loss in the inviscid limit.}
Directly substituting the limiting NLS solution \(v\) into the nearby CGL
equation produces a coefficient error of the form
\(
 (A_\eps+i)\Delta v,
\)
which is not an admissible \(\cN^1\) forcing for a general \(H^1\) solution.
We instead use
\(
 v_N=P_{\le N}v
\)
as the approximate solution in Section~\ref{sec:schrodinger}.  Here \(P_{\le N}\) is a smooth
low-frequency projection, \(v_N\to v\) strongly in \(\cX^1\), and the
nonlinear commutator
\(
 P_{\le N}f(v)-f(P_{\le N}v)
\)
also tends to zero in the critical forcing space.  For fixed \(N\), the
derivatives in the coefficient error fall on the cutoff, and
\(
 \|(A_\eps+i)\Delta v_N\|_{\cN^1(I)}
 \lesssim |A_\eps+i|\,|I|N^2
 \|v\|_{L_t^\infty H_x^1(I)}\to 0
 ~\text{as }\eps\to 0.
\)
The order of these choices is essential: first choose \(N\) large so that
the truncation and commutator errors are small; then, with \(N\) fixed,
choose \(\eps\) small so that the coefficient errors are small.  This
removes the apparent derivative loss at \(H^1\) regularity.

\smallskip
\noindent\emph{Uniform estimates for the inviscid limit.}
To obtain the uniformity required by the stability argument, write
\(
 A=a-ib,
 ~ a\ge0,
 ~ b\ge b_0>0,
 ~ a^2+b^2=1.
\)
The complex Gaussian kernel yields uniform dispersive and \(L^2\) bounds,
and hence uniform homogeneous Strichartz estimates.  The main linear step is
the coefficient-uniform retarded double-endpoint estimate
\begin{equation}\label{eq:intro-double-endpoint}
 \left\|
  \int_{\tau<t}e^{(t-\tau)A\Delta}F(\tau)\,\mathrm d\tau
 \right\|_{L_t^2L_x^{2^*}}
 \lesssim_{d,b_0}
 \|F\|_{L_t^2L_x^{(2^*)'}}.
\end{equation}
Its constant is independent of \(a\downarrow0\).
Proposition~\ref{prop:sch-double-endpoint}, which proves this estimate, is a
principal analytic ingredient.  The estimate is not a formal consequence
of the usual unitary \(TT^*\) argument.  Indeed, if
\(S_A(t)=e^{tA\Delta}\), then \(S_A\) is not unitary for \(a>0\), and
\(
 S_A(t)S_A(s)^*\ne S_A(t-s).
\)
Consequently, the estimate for
\(
 F\mapsto\int_{\tau<t}S_A(t-\tau)F(\tau)\,\mathrm d\tau
\)
does not follow directly from the standard unitary framework; for
\(L_\tau^2\to L_t^2\), the usual time-truncation argument also does
not provide the required retarded bound.

To prove the retarded estimate without relying on time truncation, we work
with the associated bilinear form.  A Whitney decomposition in time
and a spatial atomic decomposition reduce the estimate to summing over the
spatial and temporal scales, with decay when these scales are separated.
This proves
\eqref{eq:intro-double-endpoint} uniformly up to the inviscid limit.
Interpolation with the energy bounds then yields the remaining uniform
retarded estimates.  Together with the homogeneous estimates, these establish
the uniform \(\cX^1\) linear bound required by the common stability
argument in Section~\ref{sec:prelim}; see Section~\ref{sec:sch-linear}.

\smallskip
\noindent\emph{The inviscid limit.}
Combining this stability with the
frequency truncation of the NLS solution yields the qualitative inviscid
limit for \(H^1\) data.  Here the double-endpoint estimate keeps the
stability constant uniform as \(a\downarrow0\), while choosing \(N\) before
\(\eps\) avoids extra derivatives on the limiting solution.

For \(v_0\in H^{1+s}\), \(0<s\le2\), we smooth the initial data at
frequency \(N\) and apply Lemma~\ref{lem:nls-H3-persistence} to the
corresponding NLS solution.  Choosing \(N\sim|A_\eps+i|^{-1/2}\) when
\(A_\eps\ne-i\) yields the rate \(|A_\eps+i|^{s/2}\) in
Theorem~\ref{thm:intro-B}; see Section~\ref{sec:schrodinger}.

\smallskip
\noindent\emph{The zero-dispersion limit.}
At the zero-dispersion limit, uniform positivity of
\(\operatorname{Re}A_\eps\) gives the derivative control needed to estimate
\((A_\eps-1)\Delta h\) at \(H^1\) regularity.  After a linear correction,
the common stability estimate yields the quantitative limit in
Theorem~\ref{thm:intro-zero-dispersion} without higher regularity; see
Section~\ref{sec:heat}.

 \section{Preliminaries and common stability estimates}\label{sec:prelim}

Unless a smaller range is stated, $d\in\{3,4,5,6\}$ and
\[
 p=\frac4{d-2},\qquad f(u)=|u|^p u,
 \qquad 2^*=\frac{2d}{d-2},\qquad (2^*)'=\frac{2d}{d+2}.
\]
We also set
\[
 p_c=q_c=\frac{2(d+2)}{d-2},
 \qquad r_c=\frac{2d(d+2)}{d^2+4}.
\]
For \(1\le s\le\infty\), we use the abbreviation
\(
 \|g\|_s:=\|g\|_{L^s(\R^d)}.
\)
Space-time norms and norms over a specified time interval will always
carry the corresponding \(t\)- and \(x\)-subscripts.

We also fix a real-valued radial function
\(
 \chi\in C_c^\infty(\R^d),~
 0\le\chi\le1,~
 \chi(\xi)=1\ \text{if }|\xi|\le1,~
 \chi(\xi)=0\ \text{if }|\xi|\ge2,
\)
and define the smooth low-frequency cutoff \(P_{\le N}\), for \(N\ge1\),
by
\(
 \widehat{P_{\le N}g}(\xi)
 =\chi(\xi/N)\widehat g(\xi).
\)
This is the usual Littlewood--Paley notation; see, for example,
\cite{MiaoZhangZheng2025}.

The meaning of these exponents is as follows.  The exponent
\(2^*=2d/(d-2)\) is the Sobolev endpoint associated with
\(\dot H^1(\R^d)\), and \((2^*)'=2d/(d+2)\) is its dual exponent.  The
number \(p_c=2(d+2)/(d-2)\) is the scaling-critical spacetime exponent
for the energy-critical nonlinearity.  The roles of \(q_c\) and \(r_c\) in
the Strichartz and Sobolev estimates are described below.
For a finite interval $I$, define
\begin{align*}
 \|u\|_{\cX^1(I)}={}
 \|u\|_{L_t^\infty H_x^1(I\times\R^d)}
 +\|u\|_{L_t^2W_x^{1,2^*}(I\times\R^d)}
 +\|\nabla u\|_{L_t^{q_c}L_x^{r_c}(I\times\R^d)}
 +\|u\|_{L_{t,x}^{p_c}(I\times\R^d)},
\end{align*}
and the critical norm
\[
 \|u\|_{\cZ^1(I)}
 =\|\nabla u\|_{L_t^{q_c}L_x^{r_c}(I)}
 +\|u\|_{L_{t,x}^{p_c}(I)}.
\]
The forcing space is the Banach sum
\[
 \cN^1(I)=L_t^1H_x^1(I\times\R^d)
 +L_t^2W_x^{1,(2^*)'}(I\times\R^d),
\]
equipped with
\[
 \|F\|_{\cN^1(I)}
 =\inf_{F=F_1+F_2}
 \bigl(\|F_1\|_{L_t^1H_x^1(I)}
 +\|F_2\|_{L_t^2W_x^{1,(2^*)'}(I)}\bigr).
\]
Every use of this space below respects this sum decomposition; no term is
implicitly required to belong to both summands.

We use the same symbols for norms on subintervals $J\subset I$.  All four
components of $\cX^1$ decrease under restriction.  For the sum space, if
$F=F_1+F_2$ is any admissible decomposition on $I$, then restriction yields
an admissible decomposition on $J$ and hence
\[
 \|F\|_{\cN^1(J)}
 \le \|F_1\|_{L_t^1H_x^1(J)}
 +\|F_2\|_{L_t^2W_x^{1,(2^*)'}(J)}.
\]
This elementary observation is important when a finite interval is divided
into critical-small pieces: the two summands have different time
summability and must be accumulated separately.

A pair $(q,r)$ is Schr\"odinger-admissible if
\(
 2\le q,r\le\infty,
 ~2/q+d/r=d/2.
\)
The three pairs
\(
 (\infty,2),\) $ (2,2^*),$ $~(q_c,r_c)$
are admissible.  In contrast, $(p_c,p_c)$ is not admissible.  The identity
\(
 1/p_c=1/r_c-1/d,
 ~q_c=p_c,
\)
shows instead that the homogeneous Sobolev inequality yields
\begin{equation}\label{eq:sob-pc}
 \|u\|_{L_{t,x}^{p_c}(I)}
 \lesssim_d\|\nabla u\|_{L_t^{q_c}L_x^{r_c}(I)}.
\end{equation}
The corresponding sharp Sobolev inequality is due to Aubin and
Talenti~\cite{Aubin1976,Talenti1976}.

Interpolation between $L_t^\infty L_x^2$ and $L_t^2L_x^{2^*}$, with
$\theta=(d-2)/(d+2)=2/q_c$, yields
\begin{equation}\label{eq:aux-interp}
 \|w\|_{L_t^{q_c}L_x^{r_c}(J)}
 \le
 \|w\|_{L_t^\infty L_x^2(J)}^{1-\theta}
 \|w\|_{L_t^2L_x^{2^*}(J)}^\theta
 \le \|w\|_{\cX^1(J)}.
\end{equation}
Indeed,
\[
 \frac1{q_c}=\frac\theta2,
 \qquad
 \frac1{r_c}=\frac{1-\theta}{2}+\frac\theta{2^*}.
\]

Finally, since $p\ge1$, the map $f:\mathbb C\simeq\R^2\to\mathbb C$ is
$C^1$ and $Df$ is locally Lipschitz.  At every point where $D^2f$ exists,
\[
 |Df(z)|\lesssim_d |z|^p,
 \qquad |D^2f(z)|\lesssim_d |z|^{p-1}.
\]
Consequently,
\begin{align}
 |f(u)-f(v)|&\lesssim_d (|u|^p+|v|^p)|u-v|,
 \label{eq:pointwise-f}\\
 |\nabla(f(u)-f(v))|&\lesssim_d
 (|u|^p+|v|^p)|\nabla(u-v)|\notag\\
 &\quad+(|u|^{p-1}+|v|^{p-1})
 (|\nabla u|+|\nabla v|)|u-v|.
 \label{eq:pointwise-grad-f}
\end{align}
These inequalities are valid for Sobolev functions by the usual
almost-everywhere chain rule.  When $p=1$, the factors with exponent
$p-1$ are understood as $1$.

All scalar products below are the complex $L^2$ scalar product, linear in
the first entry, and energy identities are understood after taking real
parts.  Constants denoted by $C_d$ depend only on the dimension.  Other
dependencies are displayed explicitly.  Constants in the
inviscid-limit linear theory are allowed to depend on the fixed lower
bound $b_0$ but not on the dissipative coefficient $a$.
\subsection{Common nonlinear and stability estimates}

The linear estimates used below reduce the nonlinear part of the perturbation
argument to controlling \(f(u)-f(v)\) in
\(L_t^2W_x^{1,(2^*)'}\).  The norms in \(\cX^1\) are matched to this
forcing space by the identities
\begin{equation}\label{eq:critical-identities}
 \frac12=\frac p{p_c}+\frac1{q_c},
 \qquad
 \frac1{(2^*)'}=\frac p{p_c}+\frac1{r_c}
\end{equation}
which, together with
\eqref{eq:pointwise-f}--\eqref{eq:pointwise-grad-f}, imply the following
estimate.

\begin{lemma}\label{lem:nonlinear-lip}
For every finite interval $J$ and $u,v\in\cX^1(J)$,
\begin{equation}\label{eq:nonlinear-lip}
 \|f(u)-f(v)\|_{L_t^2W_x^{1,(2^*)'}(J)}
 \le C_d\bigl(\|u\|_{\cZ^1(J)}^p+
 \|v\|_{\cZ^1(J)}^p\bigr)\|u-v\|_{\cX^1(J)}.
\end{equation}
In particular,
\begin{equation}\label{eq:nonlinear-single}
 \|f(u)\|_{L_t^2W_x^{1,(2^*)'}(J)}
 \le C_d\|u\|_{\cZ^1(J)}^p\|u\|_{\cX^1(J)}.
\end{equation}
\end{lemma}

\begin{proof}
Put $h=u-v$.  H\"older's inequality and
\eqref{eq:critical-identities} imply
\[
 \|(|u|^p+|v|^p)h\|_{L_t^2L_x^{(2^*)'}}
 \lesssim
 (\|u\|_{L^{p_c}_{t,x}}^p+\|v\|_{L^{p_c}_{t,x}}^p)
 \|h\|_{L_t^{q_c}L_x^{r_c}},
\]
and the same estimate with $\nabla h$ in place of $h$.  For the cross
term in \eqref{eq:pointwise-grad-f}, split $p/p_c$ as
$(p-1)/p_c+1/p_c$ to obtain
\[
 \||a|^{p-1}|\nabla b||h|\|_{L_t^2L_x^{(2^*)'}}
 \le \|a\|_{L_{t,x}^{p_c}}^{p-1}
 \|\nabla b\|_{L_t^{q_c}L_x^{r_c}}
 \|h\|_{L_{t,x}^{p_c}},
\]
where $a, b\in \{u,v\}$.
When $p=1$, the first factor on the right is omitted.
By \eqref{eq:aux-interp}, $\|h\|_{L_t^{q_c}L_x^{r_c}}$ is bounded by
$\|h\|_{\cX^1}$.  By Young's inequality, each product
$\|a\|^{p-1}\|\nabla b\|$
can be bounded by a constant multiple of
$\|a\|_{\cZ^1}^p+\|b\|_{\cZ^1}^p$.
More explicitly, the four choices $(a,b)=(u,u),(u,v),(v,u),(v,v)$
arising from \eqref{eq:pointwise-grad-f} are bounded by
\[
 C_d\bigl(\|u\|_{\cZ^1(J)}^p+
 \|v\|_{\cZ^1(J)}^p\bigr)\|h\|_{\cX^1(J)}.
\]
Taking $v=0$ yields \eqref{eq:nonlinear-single}.
\end{proof}

Assume that, for the coefficients \(A\) and \(r\), the linear Cauchy
problem satisfies
\begin{equation}\label{eq:common-linear-assumption}
 \|z\|_{\cX^1(J)}
 \le K\bigl(\|z(\tau_0)\|_{H^1}+\|F\|_{\cN^1(J)}\bigr)
\end{equation}
on every subinterval \(J=[\tau_0,\tau_1]\subset I\), where
\(\partial_tz-A\Delta z+rz=F\) and \(K\ge1\) is independent of \(J\).
On an interval where the approximate solution is small in \(\cZ^1\),
Lemma~\ref{lem:nonlinear-lip} allows the nonlinear term to be absorbed.
A finite interval can then be divided into finitely many such
subintervals, and the error is propagated from one subinterval to the
next.  This establishes the common stability statement used in both limits.

\begin{proposition}\label{prop:common-stability}
Let \(I=[t_0,t_1]\), \(A,C\in\C\), \(\operatorname{Re}A\ge0\),
\(r\ge0\), \(C_0\ge1\), \(|C|\le C_0\), and suppose that
\eqref{eq:common-linear-assumption} holds with constant \(K\) on every
subinterval of \(I\).  Let \(M>0\), and suppose
\(U\in C(I;H^1)\cap\cX^1(I)\) satisfies
\[
 \|U\|_{\cX^1(I)}\le M,
 \qquad \partial_tU-A\Delta U+\mu C f(U)+rU=e.
\]
There are constants
\[
 \varepsilon_*=\varepsilon_*(d,K,C_0,|I|,M)>0,
 \quad K_*=K_*(d,K,C_0,|I|,M)
\]
such that, whenever
\[
 \|v_0-U(t_0)\|_{H^1}+\|e\|_{\cN^1(I)}\le\varepsilon_*,
\]
the exact equation with initial value \(v_0\) has a unique solution
\(v\in C(I;H^1)\cap\cX^1(I)\) and
\begin{equation}\label{eq:common-stability}
 \|v-U\|_{\cX^1(I)}
 \le K_*\bigl(\|v_0-U(t_0)\|_{H^1}+\|e\|_{\cN^1(I)}\bigr).
\end{equation}
\end{proposition}

\begin{proof}
Since every subinterval $J\subset I$ has length at most $|I|$, the same value of $K$ in
\eqref{eq:common-linear-assumption} is valid on each such $J$.  Choose $\delta>0$ so small that
\begin{equation}\label{eq:common-stability-delta}
 KC_0C_d\bigl(\delta^p+(2\delta)^p\bigr)\le\frac14,
 \qquad
 2KC_0C_d(2\delta)^p\le\frac12.
\end{equation}
This choice depends only on the parameters displayed in the statement.

Consider an interval $J=[\tau_0,\tau_1]\subset I$ such that
$\|U\|_{\cZ^1(J)}\le\delta$.  Writing $w=v-U$, we obtain
\begin{equation}\label{eq:common-stability-difference}
 \partial_tw-A\Delta w+rw
 =-\mu C\{f(U+w)-f(U)\}-e.
\end{equation}
The value $w(\tau_0)$ is regarded below as fixed
initial data.  Define
\begin{align*}
 \Phi_J(w)(t)={}&e^{-r(t-\tau_0)}e^{(t-\tau_0)A\Delta}w(\tau_0)
 -\int_{\tau_0}^te^{-r(t-s)}e^{(t-s)A\Delta}e(s)\,\mathrm ds\\
 &-\mu C\int_{\tau_0}^te^{-r(t-s)}e^{(t-s)A\Delta}
 \{f(U+w)-f(U)\}(s)\,\mathrm ds.
\end{align*}
By Lemma~\ref{lem:nonlinear-lip}, the last integrand belongs to
$L_t^2W_x^{1,(2^*)'}(J)$ whenever $w\in\cX^1(J)$; hence it belongs to the
second summand of $\cN^1(J)$, and the linear estimate~\eqref{eq:common-linear-assumption}
applies to every term in the definition of $\Phi_J$.

Set
\[
 E_J=\|w(\tau_0)\|_{H^1}+\|e\|_{\cN^1(J)},
 \qquad R_J=2KE_J,
\]
and assume that $R_J\le\delta$.  The closed ball
\(
 B_J=\{w\in\cX^1(J):\ \|w\|_{\cX^1(J)}\le R_J\}
\)
is complete in the metric induced by $\cX^1(J)$.  Since the
$\cZ^1(J)$ components are among the components controlled by
$\cX^1(J)$, every $w\in B_J$ satisfies
\[
 \|U+w\|_{\cZ^1(J)}
 \le \|U\|_{\cZ^1(J)}+\|w\|_{\cZ^1(J)}
 \le\delta+R_J\le2\delta.
\]
Applying the linear estimate~\eqref{eq:common-linear-assumption} to \eqref{eq:common-stability-difference}
and then Lemma~\ref{lem:nonlinear-lip}, we obtain
\begin{align*}
 \|\Phi_J(w)\|_{\cX^1(J)}
 &\le K\bigl(\|w(\tau_0)\|_{H^1}+\|e\|_{\cN^1(J)}\bigr)
 +KC_0\|f(U+w)-f(U)\|_{L_t^2W_x^{1,(2^*)'}(J)}\\
 &\le KE_J+KC_0C_d
 \bigl(\|U+w\|_{\cZ^1(J)}^p+\|U\|_{\cZ^1(J)}^p\bigr)
 \|w\|_{\cX^1(J)}\\
 &\le KE_J+KC_0C_d\bigl((2\delta)^p+\delta^p\bigr)R_J\\
 &\le \frac12R_J+\frac14R_J<R_J,
\end{align*}
where we used $R_J=2KE_J$ and \eqref{eq:common-stability-delta} in the
last line.  Thus $\Phi_J$ maps $B_J$ into itself.

For $w_1,w_2\in B_J$, the linear terms cancel and another application
of Lemma~\ref{lem:nonlinear-lip} yields
\begin{align*}
 \|\Phi_J(w_1)-\Phi_J(w_2)\|_{\cX^1(J)}
 &\le KC_0\|f(U+w_1)-f(U+w_2)\|_{L_t^2W_x^{1,(2^*)'}(J)}\\
 &\le KC_0C_d\bigl(
 \|U+w_1\|_{\cZ^1(J)}^p+
 \|U+w_2\|_{\cZ^1(J)}^p\bigr)
 \|w_1-w_2\|_{\cX^1(J)}\\
 &\le 2KC_0C_d(2\delta)^p
 \|w_1-w_2\|_{\cX^1(J)}\\
 &\le\frac12\|w_1-w_2\|_{\cX^1(J)}.
\end{align*}
Banach's fixed-point theorem therefore yields a unique fixed point in
$B_J$. It obeys
\begin{equation}\label{eq:common-local-stability}
 \|w\|_{\cX^1(J)}\le R_J=2KE_J.
\end{equation}

To iterate the local estimate, partition \(I\) into consecutive intervals
\(J_j=[t_j,t_{j+1}]\) as follows.  Starting at \(t_j\), choose
\(t_{j+1}\) as the first time at which either
\[
 \|\nabla U\|_{L_t^{q_c}L_x^{r_c}([t_j,t_{j+1}])}
 \quad\text{or}\quad
 \|U\|_{L_{t,x}^{p_c}([t_j,t_{j+1}])}
\]
reaches \(\delta/2\); if neither does, take \(t_{j+1}=\sup I\).
The absolute continuity of the two integrals ensures that
\(\|U\|_{\cZ^1(J_j)}\le\delta\).  On each interval except possibly
the last, one of these norms equals \(\delta/2\).  Since both norms
on \(I\) are bounded by \(M\), summing their respective
\(q_c\)-th and \(p_c\)-th powers bounds the number of intervals by
\(
 N\le N_0:=\left\lceil
 1+(2M/\delta)^{p_c}+(2M/\delta)^{q_c}
 \right\rceil,
\)
where \(\lceil x\rceil\) denotes the least integer greater than or equal to
\(x\) , and \(N_0\) depends only on \(d,M,\delta\).

Set \(b_j=\|e\|_{\cN^1(J_j)}\).  Restricting a near-optimal
decomposition in the Banach-sum definition of \(\cN^1(I)\), and using
Cauchy--Schwarz for its \(L_t^2\) component, yields
\[
 \sum_{j=0}^{N-1}b_j\le\sqrt{N_0}\,\|e\|_{\cN^1(I)}.
\]
Put \(d_j=\|w(t_j)\|_{H^1}\) and \(L=\max\{2K,1\}\).  As long as the
local smallness condition holds, \eqref{eq:common-local-stability} yields
\[
 d_{j+1}\le L(d_j+b_j),\qquad
 d_j\le L^{N_0}\left(d_0+\sum_{\ell=0}^{N-1}b_\ell\right).
\]
Hence \(d_j+b_j\le2L^{N_0}\sqrt{N_0}
(d_0+\|e\|_{\cN^1(I)})\).  Taking
\(
 \varepsilon_*\le\delta / 4KL^{N_0}\sqrt{N_0}
\)
ensures \(2K(d_j+b_j)\le\delta\) at every step, and the local
construction therefore iterates across \(I\).

Summing \eqref{eq:common-local-stability} over the partition now yields
\[
 \|w\|_{\cX^1(I)}
 \le2K\sqrt{N_0}\bigl(N_0L^{N_0}+1\bigr)
 \bigl(d_0+\|e\|_{\cN^1(I)}\bigr).
\]
This proves \eqref{eq:common-stability}, with \(K_*\) depending only on
\(d,K,C_0,|I|,M\).  Applying the same subdivision to the
difference of two solutions proves uniqueness in
\(C(I;H^1)\cap\cX^1(I)\).
\end{proof}

\section{Uniform linear estimates for the inviscid limit}\label{sec:sch-linear}

We now turn to the inviscid limit.  As the coefficient
$A_{\eps}\to -i$, the dissipative part may vanish, and the uniformly parabolic
energy estimate no longer provides uniform control.  The main task is
therefore to establish linear estimates whose constants remain
bounded as the real part of the coefficient tends to zero.  Once these
estimates are available, they supply the linear bound required by the
common stability estimate in Proposition~\ref{prop:common-stability}.

Fix \(b_0\in(0,1]\) and write
\begin{equation}\label{eq:A-convention}
 A=a-ib,\qquad a\ge0,\qquad b\ge b_0,\qquad a^2+b^2=1.
\end{equation}
For \(t>0\), let
\(
 S_A(t)=e^{tA\Delta}.
\)

\subsection{Homogeneous estimates}

We begin with the homogeneous theory.  The modulus condition on \(A\)
yields a uniform kernel bound, while the lower bound on \(b\) preserves the
dispersive decay in the time difference that is needed in the
Keel--Tao argument.  Thus the usual admissible estimates remain uniform
when \(a\to 0\).

\begin{lemma}\label{lem:sch-kernel-hom}
Let \(A\) satisfy \eqref{eq:A-convention}.  Then, for \(t>0\),
\begin{equation}\label{eq:sch-disp-basic}
 \|S_A(t)\|_{L^1_x\to L^\infty_x}\le C_d t^{-d/2},
\end{equation}
and
\begin{equation}\label{eq:sch-energy-basic}
 \|S_A(t)\phi\|_2\le \|\phi\|_2.
\end{equation}
Moreover, if \((q,r)\) is Schr\"odinger-admissible, then
\begin{equation}\label{eq:sch-homogeneous}
 \|S_A(t-t_0)\phi\|_{L^q_tL^r_x(I\times\R^d)}
 \le C(d,b_0,q,r)\|\phi\|_2,
\end{equation}
uniformly in the finite interval \(I\) and uniformly in \(a\ge0\).
\end{lemma}

\begin{proof}
For \(a>0\), the kernel is
\begin{equation}\label{eq:sch-complex-gaussian}
 K_A(t,x)=(4\pi At)^{-d/2}
 \exp\left(-\frac{|x|^2}{4At}\right).
\end{equation}
Both sides depend continuously on \(A\) in the space of tempered
distributions as \(a\downarrow0\), so the same formula defines the Schr\"odinger kernel at \(a=0\).  Since \(|A|=1\) and
\(A^{-1}=\overline A=a+ib\),
\[
 |K_A(t,x)|=(4\pi t)^{-d/2}
 \exp\left(-\frac{a|x|^2}{4t}\right)
 \le (4\pi t)^{-d/2},
\]
which proves \eqref{eq:sch-disp-basic}.  By Plancherel's identity,
\[
 \|S_A(t)\phi\|_2^2
 =\int e^{-2at|\xi|^2}|\widehat\phi(\xi)|^2\,\mathrm d\xi
 \le \|\phi\|_2^2,
\]
which is \eqref{eq:sch-energy-basic}.

We also record the corresponding bound without the normalization
\(|A|=1\).  Let \(D\in\C\setminus\{0\}\) satisfy
\(\operatorname{Re}D\ge0\).  With the usual branch convention, the kernel
of \(e^{D\Delta}\) is
\[
 K_D(x)=(4\pi D)^{-d/2}
 \exp\left(-\frac{|x|^2}{4D}\right),
\]
where, on \(\operatorname{Re}D=0\), the right-hand side is a bounded
oscillatory function and represents the corresponding limit in the space of
tempered distributions.  Since
\(
 \operatorname{Re}(D^{-1})=\operatorname{Re}D/|D|^2\ge0,
\)
we have
\[
 |K_D(x)|
 =(4\pi|D|)^{-d/2}
 \exp\left(-\frac{|x|^2}{4}\operatorname{Re}(D^{-1})\right)
 \le C_d|D|^{-d/2}.
\]
Consequently,
\begin{equation}\label{eq:sch-general-complex-gaussian}
 \|e^{D\Delta}\|_{L^1_x\to L^\infty_x}
 \le C_d|D|^{-d/2}.
\end{equation}

Since \(S_A(s)^*=S_{\overline A}(s)\),
\[
 S_A(t)S_A(s)^*=e^{(a(t+s)-ib(t-s))\Delta}.
\]
Writing \(D_{t,s}=a(t+s)-ib(t-s)\), one has
\[
 |D_{t,s}|^2=a^2(t+s)^2+b^2(t-s)^2\ge b_0^2|t-s|^2.
\]
For \(t\ne s\), estimate \eqref{eq:sch-general-complex-gaussian} therefore
yields
\begin{equation}\label{eq:sch-product-disp}
 \|S_A(t)S_A(s)^*\|_{L^1_x\to L^\infty_x}
 \le C_d b_0^{-d/2}|t-s|^{-d/2}.
\end{equation}
The same estimates hold with \(A\) replaced by \(\overline A\).  Set
\(
 U_B(\tau)=\mathbf 1_{\{\tau\ge0\}}S_B(\tau),~B\in\{A,\overline A\}.
\)
Equations \eqref{eq:sch-energy-basic} and \eqref{eq:sch-product-disp}
verify the conditions of the
endpoint theorem of Keel--Tao~\cite{KeelTao1998}.  Hence, for
every admissible pair \((q,r)\),
\[
 \|U_B(\tau)\phi\|_{L^q_\tau L^r_x(\R\times\R^d)}
 \le C(d,b_0,q,r)\|\phi\|_2,
\]
and the corresponding dual estimate holds.  Restriction and time
translation yield \eqref{eq:sch-homogeneous} on every finite interval.
\end{proof}

\subsection{The retarded double-endpoint estimate}

The homogeneous estimates do not by themselves imply the retarded
inhomogeneous estimate with both time exponents equal to two. The usual
non-endpoint time-truncation step is not available in this equal-exponent
case. To reach the double endpoint, we first localize the bilinear form
to separated time rectangles. On each rectangle, the semigroup can be
factored at an intermediate time. Interpolating the resulting
homogeneous bound with the dispersive bound yields estimates for spatial
exponents in a neighborhood of the endpoint.

\begin{lemma}\label{lem:sch-separated-rectangles}
Put \(\sigma=d/2\).  Let \(Q=I_\tau\times J_t\) be a time rectangle such
that \(I_\tau\) lies before \(J_t\),
\[
 |I_\tau|=|J_t|=\lambda,
 \qquad \lambda\lesssim \operatorname{dist}(I_\tau,J_t)\lesssim \lambda.
\]
For
\[
 B_{A,Q}(F,G)=\int_{J_t}\int_{I_\tau}
 \langle S_A(t-\tau)F(\tau),G(t)\rangle\,\mathrm d\tau\,\mathrm dt,
\]
there exists \(\delta=\delta(d)>0\) such that, whenever
\[
 \left|\frac1r-\frac1{2^*}\right|<\delta,
 \qquad
 \left|\frac1{\widetilde r}-\frac1{2^*}\right|<\delta,
\]
one has
\begin{equation}\label{eq:sch-separated-estimate}
 |B_{A,Q}(F,G)|
 \le C(d,b_0)\lambda^{\beta(r,\widetilde r)}
 \|F\|_{L^2_tL_x^{\widetilde r'}(I_\tau)}
 \|G\|_{L^2_tL_x^{r'}(J_t)},
\end{equation}
where
\begin{equation}\label{eq:sch-beta}
 \beta(r,\widetilde r)
 =1-\frac d2\left(1-\frac1r-\frac1{\widetilde r}\right).
\end{equation}
The constant is uniform in \(A\) satisfying \eqref{eq:A-convention}.
\end{lemma}

\begin{proof}
Choose any
\(
 c\in[\sup I_\tau,\inf J_t].
\)
For \(\tau\in I_\tau\) and \(t\in J_t\), both \(c-\tau\) and \(t-c\)
are nonnegative.  Hence the semigroup law and the identity
\(S_A(\rho)^*=S_{\overline A}(\rho)\) yield
\begin{align}
 B_{A,Q}(F,G)
 &\;=\int_{J_t}\int_{I_\tau}
 \bigl\langle S_A(t-c)S_A(c-\tau)F(\tau),G(t)\bigr\rangle
 \,\mathrm d\tau\,\mathrm dt\notag\\
 &\;=\left\langle
 \int_{I_\tau}S_A(c-\tau)F(\tau)\,\mathrm d\tau,
 \int_{J_t}S_{\overline A}(t-c)G(t)\,\mathrm dt
 \right\rangle _{L^2_x}.
 \label{eq:sch-factor-rectangle}
\end{align}
Let \((q_0,r_0)\) and
\((\widetilde q_0,\widetilde r_0)\) be Schr\"odinger-admissible pairs.
After the changes of variables \(\rho=c-\tau\) and \(\rho=t-c\), the dual
homogeneous estimates in Lemma~\ref{lem:sch-kernel-hom}, first for \(A\)
and then for \(\overline A\), yield
\begin{align*}
 \left\|\int_{I_\tau}S_A(c-\tau)F(\tau)\,\mathrm d\tau\right\|_{L^2_x}
 &\le C(d,b_0,\widetilde q_0,\widetilde r_0)
 \|F\|_{L_t^{\widetilde q_0'}L_x^{\widetilde r_0'}(I_\tau)},\\
 \left\|\int_{J_t}S_{\overline A}(t-c)G(t)\,\mathrm dt\right\|_{L^2_x}
 &\le C(d,b_0,q_0,r_0)
 \|G\|_{L_t^{q_0'}L_x^{r_0'}(J_t)}.
\end{align*}
Indeed, the first integral is the adjoint of the homogeneous map
\(\phi\mapsto S_{\overline A}(\rho)\phi\), restricted to the interval \(c-I_\tau\), while the second is the adjoint of
\(\phi\mapsto S_A(\rho)\phi\), restricted to \(J_t-c\).  Combining these
bounds with Cauchy--Schwarz in \(L^2_x\) yields the admissible corner
estimate
\begin{equation}\label{eq:sch-admissible-corner}
 |B_{A,Q}(F,G)|
 \le C
 \|F\|_{L_t^{\widetilde q_0'}L_x^{\widetilde r_0'}(I_\tau)}
 \|G\|_{L_t^{q_0'}L_x^{r_0'}(J_t)}.
\end{equation}

By $\lambda\lesssim \operatorname{dist}(I_\tau,J_t)\lesssim \lambda$,
\eqref{eq:sch-disp-basic}, and H\"older's inequality in space,
\begin{align*}
 \bigl|\langle S_A(t-\tau)F(\tau),G(t)\rangle\bigr|
 &\le \|S_A(t-\tau)F(\tau)\|_{L^\infty_x}
       \|G(t)\|_{L^1_x}\\
 &\le C_d\lambda^{-\sigma}
       \|F(\tau)\|_{L^1_x}\|G(t)\|_{L^1_x}.
\end{align*}
Integration over the rectangle yields
\begin{equation}\label{eq:sch-dispersive-corner}
 |B_{A,Q}(F,G)|
 \le C_d\lambda^{-\sigma}
 \|F\|_{L^1_tL^1_x(I_\tau)}
 \|G\|_{L^1_tL^1_x(J_t)}.
\end{equation}

We now choose the interpolation parameters.  Since
\[
 \frac1{2^*}=\frac{\sigma-1}{2\sigma}<\frac12,
\]
we may fix \(\delta>0\), depending only on \(d\), so small that
\begin{equation}\label{eq:sch-delta-choice}
 0<\delta<\frac1{2^*},
 \qquad
 2\left(\frac1{2^*}+\delta\right)
 <\min\left\{1,
 1-2^* \delta\right\}.
\end{equation}
Suppose that \(r\) and \(\widetilde r\) satisfy the hypotheses of the
lemma, and set
\begin{equation}\label{eq:sch-theta-choice}
 \theta
 =\min\left\{1,
 2^*\min\{1/r,1/\widetilde r\}\right\}.
\end{equation}
The definition of \(\theta\) implies
\[
 \theta\le \frac{2^*}{r},
 \qquad
 \theta\le \frac{2^*}{\widetilde r},
\]
and hence
\begin{equation}\label{eq:sch-spatial-lower}
 \frac1{2^*}\le \frac{1/r}{\theta},
 \qquad
 \frac1{2^*}\le \frac{1/\widetilde r}{\theta}.
\end{equation}
On the other hand, \eqref{eq:sch-delta-choice} implies
\[
 \theta
 \ge \min\left\{1,
 1-2^* \delta\right\}
 >2\left(\frac1{2^*}+\delta\right)
 \ge 2\max\left\{\frac1r,\frac1{\widetilde r}\right\}.
\]
Consequently,
\begin{equation}\label{eq:sch-spatial-upper}
 \frac{1/r}{\theta}<\frac12,
 \qquad
 \frac{1/\widetilde r}{\theta}<\frac12.
\end{equation}
Define spatial exponents \(r_0\) and \(\widetilde r_0\) by
\begin{equation}\label{eq:sch-r0-choice}
 \frac1{r_0}=\frac{1/r}{\theta},
 \qquad
 \frac1{\widetilde r_0}=\frac{1/\widetilde r}{\theta},
\end{equation}
and then define \(q_0\) and \(\widetilde q_0\) through the admissibility
relations
\begin{equation}\label{eq:sch-q0-choice}
 \frac1{q_0}=\sigma\left(\frac12-\frac1{r_0}\right),
 \qquad
 \frac1{\widetilde q_0}
 =\sigma\left(\frac12-\frac1{\widetilde r_0}\right).
\end{equation}
Equations \eqref{eq:sch-spatial-lower}--\eqref{eq:sch-spatial-upper} show
that
\(
 2<r_0,\widetilde r_0\le 2^*,
 ~
 2\le q_0,\widetilde q_0<\infty.
\)
Thus both pairs in \eqref{eq:sch-q0-choice} are admissible.  Moreover, by
shrinking \(\delta\) once more if necessary, all four exponents remain in
a fixed compact subset of the admissible range.  The constants in
\eqref{eq:sch-admissible-corner} are therefore bounded by a constant
depending only on \(d\) and \(b_0\).

By interpolation between \eqref{eq:sch-dispersive-corner} and
\eqref{eq:sch-admissible-corner} with parameter \(\theta\), we obtain
\begin{equation}\label{eq:sch-interpolated-rectangle}
 |B_{A,Q}(F,G)|
 \le C(d,b_0)\lambda^{-\sigma(1-\theta)}
 \|F\|_{L_t^{p_F}L_x^{\widetilde r'}(I_\tau)}
 \|G\|_{L_t^{p_G}L_x^{r'}(J_t)},
\end{equation}
where
\begin{equation}\label{eq:sch-time-interpolation}
 \frac1{p_F}
 =(1-\theta)+\frac{\theta}{\widetilde q_0'}
 =1-\frac{\theta}{\widetilde q_0},
 \qquad
 \frac1{p_G}
 =(1-\theta)+\frac{\theta}{q_0'}
 =1-\frac{\theta}{q_0}.
\end{equation}
The spatial exponents in \eqref{eq:sch-interpolated-rectangle} are exactly
the desired ones, because \eqref{eq:sch-r0-choice} implies
\[
 (1-\theta)+\frac{\theta}{\widetilde r_0'}
 =1-\frac{\theta}{\widetilde r_0}
 =1-\frac1{\widetilde r}
 =\frac1{\widetilde r'},
\]
and similarly
\[
 (1-\theta)+\frac{\theta}{r_0'}=\frac1{r'}.
\]
Since \(q_0,\widetilde q_0\ge2\) and \(0<\theta\le1\),
\(p_F,p_G\le2\).  H\"older inequality on the intervals of length
\(\lambda\) therefore yields
\begin{align}
 \|F\|_{L_t^{p_F}L_x^{\widetilde r'}(I_\tau)}
 &\le \lambda^{1/p_F-1/2}
 \|F\|_{L_t^2L_x^{\widetilde r'}(I_\tau)},\label{eq:sch-time-holder-F}\\
 \|G\|_{L_t^{p_G}L_x^{r'}(J_t)}
 &\le \lambda^{1/p_G-1/2}
 \|G\|_{L_t^2L_x^{r'}(J_t)}.
 \label{eq:sch-time-holder-G}
\end{align}
Using \eqref{eq:sch-q0-choice} and \eqref{eq:sch-r0-choice}, the two time
powers are
\begin{align*}
 \frac1{p_F}-\frac12
 &=\frac12-\frac{\theta}{\widetilde q_0}
 =\frac12-\frac{\sigma\theta}{2}
   +\frac{\sigma}{\widetilde r},\\
 \frac1{p_G}-\frac12
 &=\frac12-\frac{\theta}{q_0}
 =\frac12-\frac{\sigma\theta}{2}
   +\frac{\sigma}{r}.
\end{align*}
Substituting \eqref{eq:sch-time-holder-F}--\eqref{eq:sch-time-holder-G}
into \eqref{eq:sch-interpolated-rectangle}, the total power of \(\lambda\)
is
\begin{align*}
 &-\sigma(1-\theta)
 +\left(\frac12-\frac{\sigma\theta}{2}
        +\frac{\sigma}{\widetilde r}\right)
 +\left(\frac12-\frac{\sigma\theta}{2}
        +\frac{\sigma}{r}\right)\\
 &\qquad
 =1-\sigma+\frac{\sigma}{r}+\frac{\sigma}{\widetilde r}
 =1-\frac d2\left(1-\frac1r-\frac1{\widetilde r}\right)
 =\beta(r,\widetilde r).
\end{align*}
This proves \eqref{eq:sch-separated-estimate}. 
\end{proof}

At \(r=\widetilde r=2^*\), the exponent
\(\beta(r,\widetilde r)\) vanishes, so the preceding estimate is
scale-invariant on each separated rectangle.  A direct sum over all time
scales would not provide the decay needed at the endpoint.  The next
proposition is the main linear estimate of this subsection.  Its proof
combines a Whitney decomposition in time with a spatial atomic
decomposition.  For each pair of time and spatial scales, the exponents
in Lemma~\ref{lem:sch-separated-rectangles} are chosen on the appropriate
sides of \(2^*\).  This turns the mismatch between the two scales into a
summable factor and yields a constant that is uniform as \(a\downarrow0\).

\begin{proposition}
\label{prop:sch-double-endpoint}
Let \(A\) satisfy \eqref{eq:A-convention}.  Then
\begin{equation}\label{eq:sch-double-endpoint}
 \left\|\int_{\tau<t}S_A(t-\tau)F(\tau)\,\mathrm d\tau\right\|_{L^2_tL^{2^*}_x(\R\times\R^d)}
 \le C(d,b_0)\|F\|_{L^2_tL^{(2^*)'}_x(\R\times\R^d)}.
\end{equation}
\end{proposition}

\begin{proof}
By duality, it suffices to prove
\begin{equation}\label{eq:sch-bilinear-double}
 |B_A(F,G)|\le C(d,b_0)
 \|F\|_{L^2_tL^{(2^*)'}_x}\|G\|_{L^2_tL^{(2^*)'}_x},
\end{equation}
where
\[
 B_A(F,G)=\iint_{\tau<t}
 \langle S_A(t-\tau)F(\tau),G(t)\rangle\,\mathrm d\tau\,\mathrm dt.
\]
We begin with \(F\) and \(G\) compactly supported in time and simple as
functions of time, with bounded, compactly supported spatial values.  In
particular, they belong to \(L^1_tL^2_x\).  The \(L^2\)-contraction of
\(S_A\) implies
\[
 \iint_{\tau<t}
 \bigl|\langle S_A(t-\tau)F(\tau),G(t)\rangle\bigr|\,\mathrm d\tau\,\mathrm dt
 \le \|F\|_{L^1_tL^2_x}\|G\|_{L^1_tL^2_x}<\infty,
\]
so all decompositions and rearrangements made below are justified first
for this dense class.

\medskip
\noindent\emph{Step 1: Whitney decomposition of time intervals.}
Let
\(
 \Omega=\{(\tau,t)\in\R^2:\tau<t\}.
\)
A dyadic Whitney decomposition of \(\Omega\) provides a countable family
\(\mathcal Q\) of essentially disjoint dyadic squares
\(
 Q=I_Q\times J_Q\subset\Omega
\)
whose union is \(\Omega\) up to a null set and such that, if
\(\lambda_Q\) denotes the common side length of \(I_Q\) and \(J_Q\), then
\begin{equation}\label{eq:sch-whitney-geometry}
 \lambda_Q\lesssim \operatorname{dist}(I_Q,J_Q)
 \lesssim \lambda_Q.
\end{equation}
This is the standard Whitney decomposition of an open set; see
Stein~\cite{Stein1970}.  Since the
squares are dyadic, \(\lambda_Q=2^k\) for some \(k\in\mathbb Z\).  Write
\(
 \mathcal Q_k=\{Q\in\mathcal Q:\lambda_Q=2^k\}.
\)
The \eqref{eq:sch-whitney-geometry} has the following consequence: there is an absolute integer \(N_0\) such that,
for every fixed \(k\), each dyadic interval of length \(2^k\) occurs as an
\(I_Q\) for at most \(N_0\) squares in \(\mathcal Q_k\), and similarly
each interval occurs as a \(J_Q\) for at most \(N_0\) squares.  Indeed,
for a fixed \(I_Q\), every admissible \(J_Q\) has the same length and must
lie in an interval of length \(O(2^k)\) determined by
\eqref{eq:sch-whitney-geometry}; only boundedly many dyadic intervals of
length \(2^k\) can do so. Therefore
\begin{equation}\label{eq:sch-whitney-bilinear-sum}
 B_A(F,G)=\sum_{Q\in\mathcal Q}B_{A,Q}(F,G).
\end{equation}

\medskip
\noindent\emph{Step 2: spatial atom decomposition at each time.}
Put
\(
 p_*:=(2^*)'=\frac{2d}{d+2}.
\)
For almost every time, decompose
\begin{equation}\label{eq:sch-atoms}
 F(t)=\sum_{m\in\mathbb Z}a_m(t)\phi_m(t),
 \qquad
 G(t)=\sum_{n\in\mathbb Z}b_n(t)\psi_n(t),
\end{equation}
where the coefficients are nonnegative, the spatial supports of 
atoms are pairwise disjoint within each decomposition, and
\begin{align}
 |\operatorname{supp}\phi_m(t)|\le2^m,~
 \|\phi_m(t)\|_{L^\infty_x}\le2^{-m/p_*},~
 |\operatorname{supp}\psi_n(t)|\le2^n,~
 \|\psi_n(t)\|_{L^\infty_x}\le2^{-n/p_*}.
 \label{eq:sch-G-atom-size}
\end{align}
Moreover,
\begin{align}
 \sum_m a_m(t)^{p_*}\le C_{p_*}\|F(t)\|_{L^{p_*}_x}^{p_*},~
 \sum_n b_n(t)^{p_*}\le C_{p_*}\|G(t)\|_{L^{p_*}_x}^{p_*}.
 \label{eq:sch-atom-coeff-F}
\end{align}
This is the dyadic atomic decomposition used in the endpoint Strichartz estimates
of Keel--Tao~\cite{KeelTao1998}.

The coefficient functions and atoms in \eqref{eq:sch-atoms} may be chosen
measurably in time.  Indeed, after choosing a disjoint representation, write
\[
 F(t)=\sum_{\ell=1}^{L_F}\mathbf1_{E_\ell}(t)f_\ell,
\]
where the sets \(E_\ell\) are measurable and the \(f_\ell\) are the finitely
many spatial values of \(F\).  For each \(f_\ell\), fix one atomic
decomposition
\(
 f_\ell=\sum_m a_{\ell,m}\phi_{\ell,m}
\)
and use it for every \(t\in E_\ell\).  Thus on \(E_\ell\) we set
\(
 a_m(t)=a_{\ell,m}
\)
and
\(
 \phi_m(t)=\phi_{\ell,m}.
\)
These functions are piecewise constant in time and hence measurable.  The
same construction applies to \(G\), \(b_n\), and \(\psi_n\). The size and support conditions imply, for every
\(1\le\rho\le\infty\),
\begin{align}
 \|\phi_m(t)\|_{L^\rho_x}
 \le2^{m(1/\rho-1/p_*)},~
 \|\psi_n(t)\|_{L^\rho_x}
 \le2^{n(1/\rho-1/p_*)}.
 \label{eq:sch-atom-Lrho-F}
\end{align}

We first perform the remaining argument with finite sums.  Choose increasing
finite subcollections \(\mathcal Q^{(L)}\subset\mathcal Q\) such that
\(
 \bigcup_{L\ge1}\mathcal Q^{(L)}=\mathcal Q,
\)
and put
\[
 F_M(t)=\sum_{|m|\le M}a_m(t)\phi_m(t),
 \qquad
 G_M(t)=\sum_{|n|\le M}b_n(t)\psi_n(t).
\]
For \(L,M\ge1\), define
\[
 B_A^{L,M}
 :=\sum_{Q\in\mathcal Q^{(L)}}B_{A,Q}(F_M,G_M).
\]
All sums in Steps 3 and 4 are first understood with these truncations.  The
estimates below are independent of \(L\) and \(M\).

\medskip
\noindent\emph{Step 3: estimate one pair of atoms on one Whitney square.}
Let \(\delta_0>0\) be the neighborhood size supplied by
Lemma~\ref{lem:sch-separated-rectangles}, and fix
\(
 0<\delta_1<\delta_0.
\)
For \(Q\in\mathcal Q_k\), Lemma~\ref{lem:sch-separated-rectangles}, and
\eqref{eq:sch-atom-Lrho-F} imply
\begin{align}
 |B_{A,Q}(a_m\phi_m,b_n\psi_n)|
 &\le C(d,b_0)\lambda_Q^{\beta(r,\widetilde r)}
 2^{m(1/\widetilde r'-1/p_*)}
 2^{n(1/r'-1/p_*)}
 \|a_m\|_{L^2(I_Q)}\|b_n\|_{L^2(J_Q)}\notag\\
 &=C(d,b_0)\lambda_Q^{\beta(r,\widetilde r)}
 2^{m(\frac1{2^*}-1/\widetilde r)}
 2^{n(\frac1{2^*}-1/r)}
 \|a_m\|_{L^2(I_Q)}\|b_n\|_{L^2(J_Q)}.
 \label{eq:sch-atom-prechoice}
\end{align}
The factor in \eqref{eq:sch-atom-prechoice} is exactly
\begin{align}
 \lambda_Q^{\beta(r,\widetilde r)}
 2^{m(\frac1{2^*}-1/\widetilde r)}
 2^{n(\frac1{2^*}-1/r)}
 =
 \left(\frac{2^m}{\lambda_Q^\sigma}\right)^{\frac1{2^*}-1/\widetilde r}
 \left(\frac{2^n}{\lambda_Q^\sigma}\right)^{\frac1{2^*}-1/r}.
 \label{eq:sch-scale-factor}
\end{align}

We now choose \(r\) and \(\widetilde r\) separately for each triple
\((m,n,Q)\).  If \(2^m\ge\lambda_Q^\sigma\), take
\(1/\widetilde r=\frac1{2^*}+\delta_1\); if
\(2^m<\lambda_Q^\sigma\), take
\(1/\widetilde r=\frac1{2^*}-\delta_1\).  Likewise, if
\(2^n\ge\lambda_Q^\sigma\), take \(1/r=\frac1{2^*}+\delta_1\), and if
\(2^n<\lambda_Q^\sigma\), take \(1/r=\frac1{2^*}-\delta_1\).  All four possible
choices are allowed by Lemma~\ref{lem:sch-separated-rectangles}.  With
\([x]_*:=\max\{x,x^{-1}\}\), each factor in
\eqref{eq:sch-scale-factor} becomes a decaying factor, and hence
\begin{equation}\label{eq:sch-atom-rectangle}
 |B_{A,Q}(a_m\phi_m,b_n\psi_n)|
 \le C(d,b_0)
 \left[\frac{2^m}{\lambda_Q^\sigma}\right]_*^{-\delta_1}
 \left[\frac{2^n}{\lambda_Q^\sigma}\right]_*^{-\delta_1}
 \|a_m\|_{L^2(I_Q)}\|b_n\|_{L^2(J_Q)}.
\end{equation}

\medskip
\noindent\emph{Step 4: summation}
Put
\(
 \mathcal Q_k^{(L)}=\mathcal Q_k\cap\mathcal Q^{(L)}.
\)
For fixed \(k,m,n\), Cauchy--Schwarz in \(Q\) gives
\begin{align}
 \sum_{Q\in\mathcal Q_k^{(L)}}
 \|a_m\|_{L^2(I_Q)}\|b_n\|_{L^2(J_Q)}
 &\le
 \left(\sum_{Q\in\mathcal Q_k^{(L)}}\|a_m\|_{L^2(I_Q)}^2\right)^{1/2}
 \left(\sum_{Q\in\mathcal Q_k^{(L)}}\|b_n\|_{L^2(J_Q)}^2\right)^{1/2}\notag\\
 &\le N_0
 \|a_m\|_{L^2_t}\|b_n\|_{L^2_t}.
 \label{eq:sch-fixed-scale-sum}
\end{align}
Indeed,
\[
 \sum_{Q\in\mathcal Q_k^{(L)}}\|a_m\|_{L^2(I_Q)}^2
 =\int_\R |a_m(t)|^2
 \sum_{Q\in\mathcal Q_k^{(L)}}\mathbf1_{I_Q}(t)\,\mathrm dt
 \le N_0\|a_m\|_{L^2_t}^2,
\]
and the estimate for \(b_n\) is identical.

Since \(\lambda_Q=2^k\), define
\(
 w_{m,n}(k)
 =2^{-\delta_1|m-\sigma k|}
  2^{-\delta_1|n-\sigma k|}.
\)
This is exactly the product of the two bracketed factors in
\eqref{eq:sch-atom-rectangle}.  We claim that
\begin{equation}\label{eq:sch-scale-convolution}
 \sum_{k\in\mathbb Z}w_{m,n}(k)
 \le C_{d,\delta_1}
 (1+|m-n|)2^{-\delta_1|m-n|}
 =:c_{m-n}.
\end{equation}
To verify this, suppose first that \(m\le n\).  If \(\sigma k\le m\), then
\[
 |m-\sigma k|+|n-\sigma k|
 =(n-m)+2(m-\sigma k),
\]
so the sum over these \(k\) is bounded by a convergent geometric series
times \(2^{-\delta_1(n-m)}\).  If
\(m\le\sigma k\le n\), then
\(
 |m-\sigma k|+|n-\sigma k|=n-m;
\)
there are at most \(C_d(1+n-m)\) such integers \(k\), and their total
contribution is bounded by
\(C_d(1+n-m)2^{-\delta_1(n-m)}\).  Finally, if \(\sigma k\ge n\), the
same geometric series argument as in the first region applies.  This
proves \eqref{eq:sch-scale-convolution} when \(m\le n\); the other case is
symmetric.  In particular,
\(
 c\in\ell^1(\mathbb Z).
\)

Combining
\eqref{eq:sch-atom-rectangle},
\eqref{eq:sch-fixed-scale-sum}, and
\eqref{eq:sch-scale-convolution} yields, uniformly in \(L\) and \(M\),
\begin{equation}\label{eq:sch-final-atom-sum}
 |B_A^{L,M}|
 \le C(d,b_0)\sum_{|m|,|n|\le M}
 c_{m-n}\|a_m\|_{L^2_t}\|b_n\|_{L^2_t}.
\end{equation}
Young's convolution inequality on \(\mathbb Z\) and Cauchy--Schwarz give
\begin{equation}\label{eq:sch-sequence-young}
 \sum_{|m|,|n|\le M}c_{m-n}\|a_m\|_{L^2_t}\|b_n\|_{L^2_t}
 \le \|c\|_{\ell^1}
 \left(\sum_{|m|\le M}\|a_m\|_{L^2_t}^2\right)^{1/2}
 \left(\sum_{|n|\le M}\|b_n\|_{L^2_t}^2\right)^{1/2}.
\end{equation}
Because \(p_*<2\), the embedding
\(\ell^{p_*}\hookrightarrow\ell^2\), Tonelli's theorem, and
\eqref{eq:sch-atom-coeff-F} imply
\begin{align*}
 \sum_m\|a_m\|_{L^2_t}^2
 &=\sum_m\int_\R a_m(t)^2\,\mathrm dt
 =\int_\R\sum_m a_m(t)^2\,\mathrm dt\\
 &\le\int_\R
 \left(\sum_m a_m(t)^{p_*}\right)^{2/p_*}\,\mathrm dt
 \le C_{p_*}\int_\R\|F(t)\|_{L^{p_*}_x}^2\,\mathrm dt.
\end{align*}
Thus
\begin{equation}\label{eq:sch-a-coefficient-bound}
 \left(\sum_m\|a_m\|_{L^2_t}^2\right)^{1/2}
 \le C_{p_*}\|F\|_{L^2_tL^{p_*}_x}.
\end{equation}
The same argument yields
\begin{equation}\label{eq:sch-b-coefficient-bound}
 \left(\sum_n\|b_n\|_{L^2_t}^2\right)^{1/2}
 \le C_{p_*}\|G\|_{L^2_tL^{p_*}_x}.
\end{equation}
It follows from \eqref{eq:sch-final-atom-sum}--
\eqref{eq:sch-b-coefficient-bound} that
\begin{equation}\label{eq:sch-finite-truncation-bound}
 |B_A^{L,M}|
 \le C(d,b_0)
 \|F\|_{L^2_tL^{p_*}_x}\|G\|_{L^2_tL^{p_*}_x},
\end{equation}
with a constant independent of \(L\) and \(M\).

We now pass to the full decompositions.  For fixed \(M\), the absolute
integrability established at the beginning of the proof, which also applies
to \(F_M\) and \(G_M\), allows
\(\mathcal Q^{(L)}\) to exhaust \(\mathcal Q\), so that
\[
 B_A^{L,M}\to B_A(F_M,G_M)
 \qquad\text{as }L\to\infty.
\]
For the chosen dense class, each spatial value is bounded and compactly
supported.  Hence the finite atomic sums converge to the corresponding
spatial values in both \(L^{p_*}_x\) and \(L^2_x\).  Since only finitely many
spatial values occur in time, it follows that
\[
 F_M\to F,\qquad G_M\to G
 \quad\text{in }L^2_tL^{p_*}_x\cap L^1_tL^2_x.
\]
The \(L^2\)-contraction of \(S_A\) and the absolute integrability of the
bilinear form therefore imply
\(
 B_A(F_M,G_M)\to B_A(F,G).
\)
Moreover, \(c\in\ell^1(\mathbb Z)\), together with
\eqref{eq:sch-a-coefficient-bound} and
\eqref{eq:sch-b-coefficient-bound}, shows that the atomic scale sums converge
absolutely.  We may thus first let \(L\to\infty\) and then \(M\to\infty\) in
\eqref{eq:sch-finite-truncation-bound}, obtaining
\eqref{eq:sch-bilinear-double} on the dense class.

Finally, compactly time-supported simple functions with bounded,
compactly supported spatial values are dense in
\(L^2_tL^{(2^*)'}_x\).  Hence \eqref{eq:sch-bilinear-double} extends by
density to arbitrary \(F,G\in L^2_tL^{(2^*)'}_x\).  For fixed \(F\), the
bilinear form defines a bounded linear functional on
\(L^2_tL^{(2^*)'}_x\); since the dual of this space is \(L^2_tL^{2^*}_x\), it is
represented by a unique element of \(L^2_tL^{2^*}_x\), which agrees with the
retarded Duhamel integral on the dense class.  This is precisely
\eqref{eq:sch-double-endpoint}.
\end{proof}

\subsection{Retarded estimates and the linear bound}

Proposition~\ref{prop:sch-double-endpoint} controls precisely the
retarded map from \(L_t^2L_x^{(2^*)'}\) to \(L_t^2L_x^{2^*}\).
Interpolating this endpoint bound with the fixed-time \(L^2\) bound yields
every admissible output pair for the same input space.  A second
interpolation with the \(L_t^1L_x^2\) estimate then yields the full range
of admissible input pairs.

\begin{lemma}\label{lem:sch-retarded-admissible}
Let
\[
 (T_AG)(t)=\int_{\tau<t}S_A(t-\tau)G(\tau)\,\mathrm d\tau.
\]
For every admissible output pair \((q,r)\) and every admissible input pair
\((\widetilde q,\widetilde r)\),
\begin{equation}\label{eq:sch-retarded-all}
 \|T_AG\|_{L^q_tL^r_x}
 \le C(d,b_0,q,r,\widetilde q,\widetilde r)
 \|G\|_{L_t^{\widetilde q'}L_x^{\widetilde r'}}.
\end{equation}
The constant is uniform in \(A\) satisfying \eqref{eq:A-convention}.
\end{lemma}

\begin{proof}
Proposition \ref{prop:sch-double-endpoint} provides
\(
 T_A:L^2_tL^{(2^*)'}_x\to L^2_tL^{2^*}_x.
\)
For each fixed \(t\), the estimate from
Lemma~\ref{lem:sch-kernel-hom}, after the change of variables
\(\rho=t-\tau\), yields
\(
 \|T_AG(t)\|_2\lesssim \|G\|_{L^2_tL^{(2^*)'}_x}.
\)
Interpolating these two estimates yields
\begin{equation}\label{eq:sch-retarded-from-endpoint}
 \|T_AG\|_{L^q_tL^r_x}
 \lesssim \|G\|_{L^2_tL^{(2^*)'}_x}
\end{equation}
for every admissible output pair \((q,r)\).  On the other hand, by
Minkowski's inequality and the homogeneous estimate
\eqref{eq:sch-homogeneous},
\begin{equation}\label{eq:sch-retarded-L1L2}
 \|T_AG\|_{L^q_tL^r_x}
 \lesssim \|G\|_{L^1_tL^2_x}
\end{equation}
for every admissible output pair.  Interpolating the two domain estimates
\eqref{eq:sch-retarded-from-endpoint} and \eqref{eq:sch-retarded-L1L2}
with parameter \(\eta=2/\widetilde q\) yields
\[
 \frac1{\widetilde q'}=1-\frac\eta2,
 \qquad
 \frac1{\widetilde r'}=\frac{1-\eta}{2}+\frac\eta{(2^*)'}.
\]
These identities are equivalent to the admissibility of
\((\widetilde q,\widetilde r)\).  Interpolation between the input spaces
\(L_t^1L_x^2\) and \(L_t^2L_x^{(2^*)'}\), with the output space
\(L_t^qL_x^r\) fixed, yields \eqref{eq:sch-retarded-all}.
\end{proof}

The preceding estimates can now be assembled in the form needed for the
nonlinear argument.  The forcing is split according to the two summands
of \(\cN^1\).  The \(L_t^1H_x^1\) part is handled by the homogeneous
estimate and Minkowski's inequality, while the
\(L_t^2W_x^{1,(2^*)'}\) part is controlled by the retarded estimates above.
Since spatial derivatives commute with \(S_A(t)\), these bounds establish all
derivative components of \(\cX^1\); the remaining critical spacetime norm
follows from the Sobolev embedding.

\begin{proposition}
       \label{prop:sch-linear}
       Let \(I=[t_0,t_1]\), let \(A\) satisfy \eqref{eq:A-convention}, and let
       \(0\le r\le r_0\).  If
       \[
        \partial_tz-A\Delta z+rz=F,
        \qquad z(t_0)=z_0\in H^1(\R^d),
       \]
       then
       \begin{equation}\label{eq:sch-linear}
        \|z\|_{\cX^1(I)}
        \le C(d,b_0)
        \bigl(\|z_0\|_{H^1}+\|F\|_{\cN^1(I)}\bigr).
       \end{equation}
       The constant is uniform in \(a\ge0\) and hence remains bounded as
       \(A\to-i\).
       \end{proposition}
       
       \begin{proof}
       We first assume \(r=0\) and write \(F=F_1+F_2\), where
       \(
        F_1\in L^1_tH^1_x(I),~ F_2\in L^2_tW^{1,(2^*)'}_x(I).
       \)
       Duhamel's formula yields
       \[
        z(t)=S_A(t-t_0)z_0+
        \int_{t_0}^tS_A(t-\tau)(F_1+F_2)(\tau)\,\mathrm d\tau.
       \]
       The three pairs
       \(
        (\infty,2),~ (2,2^*),~(q_c,r_c)
       \)
       are admissible.  The
       \(L^{p_c}_{t,x}\) component of \(\cX^1\) is obtained instead from
       \eqref{eq:sob-pc}.
       
       Apply the homogeneous estimate \eqref{eq:sch-homogeneous} to the initial
       term, the estimate \eqref{eq:sch-retarded-L1L2} to \(F_1\), and
       the estimate \eqref{eq:sch-retarded-all} with input pair \((2,2^*)\)
       to \(F_2\).  Apply the same estimates to each spatial derivative
       \(\partial_j z\), with initial data \(\partial_j z_0\) and forcing
       \(\partial_j F_1+\partial_j F_2\).  The admissible output pairs
       \((\infty,2)\), \((2,2^*)\), and \((q_c,r_c)\) therefore control the
       \(L^\infty_tH^1_x\), \(L^2_tW^{1,2^*}_x\), and
       \(L^{q_c}_tL^{r_c}_x\) gradient components of \(\cX^1\).  The remaining
       \(L^{p_c}_{t,x}\) component follows from \eqref{eq:sob-pc}.  Hence
       \[
        \|z\|_{\cX^1(I)}
        \le C(d,b_0)
        \left(\|z_0\|_{H^1}
        +\|F_1\|_{L^1_tH^1_x(I)}
        +\|F_2\|_{L^2_tW^{1,(2^*)'}_x(I)}\right).
       \]
       Taking the infimum over all decompositions \(F=F_1+F_2\) yields the
       estimate for \(r=0\) with the \(\cN^1\) norm.
       
       For \(0\le r\le r_0\), write
       \(
        S_{A,r}(t)=e^{-rt}S_A(t).
       \)
       Since \(e^{-rt}\le1\) for \(t\ge0\), the \(L^2\) and dispersive
       estimates used above remain valid for \(S_{A,r}\), with the same
       constants.  Moreover,
       \[
        S_{A,r}(t)S_{A,r}(s)^*
        =e^{-r(t+s)}S_A(t)S_A(s)^*,
       \]
       so the homogeneous estimates are unchanged.  The separated-interval
       argument in Lemma~\ref{lem:sch-separated-rectangles} and the retarded
       estimates in Proposition~\ref{prop:sch-double-endpoint} and
       Lemma~\ref{lem:sch-retarded-admissible} also apply with \(S_A\) replaced by
       \(S_{A,r}\), because the extra scalar factor is bounded by one and
       \(S_{A,r}\) has the same semigroup property.  Repeating the preceding
       Duhamel estimates therefore yields
       \[
        \|z\|_{\cX^1(I)}
        \le C(d,b_0)
        \bigl(\|z_0\|_{H^1}+\|F\|_{\cN^1(I)}\bigr),
       \]
       which is \eqref{eq:sch-linear}.
\end{proof}

\begin{remark}
If \(A_\eps\to-i\), \(|A_\eps|=1\), and \(\operatorname{Re}A_\eps\ge0\), then writing
\(A_\eps=a_\eps-ib_\eps\) implies \(b_\eps\to1\).  Hence \(b_\eps\ge1/2\) for
small \(\eps\).  If also \(r_\eps\to0\), Proposition~\ref{prop:sch-linear}
with \(b_0=1/2\) and \(r_0=1\) is uniform on every fixed finite interval.
\end{remark}

\section{The inviscid limit}\label{sec:schrodinger}

We now apply the uniform linear estimates of
Section~\ref{sec:sch-linear} and the common stability estimate of
Section~\ref{sec:prelim} to the inviscid limit.

\subsection{Low-frequency approximation and nonlinear truncation}
\label{subsec:sch-frequency-truncation}

The low-frequency cutoff is used in the
inviscid limit to avoid derivative loss.  All constants below may depend
on \(\chi\), but not on \(N\).

\begin{lemma}
\label{lem:frequency-truncation}
Let \(I\) be a finite interval and let
\(
 U\in C(I;H^1)\cap\cX^1(I).
\)
Set \(U_N=P_{\le N}U\).  Then
\begin{equation}\label{eq:PN-uniform-X}
 \sup_{N\ge1}\|U_N\|_{\cX^1(I)}
 \le C_\chi\|U\|_{\cX^1(I)},
\end{equation}
and
\begin{equation}\label{eq:PN-strong-X}
 \|U_N-U\|_{\cX^1(I)}\to0
 \qquad\text{as }N\to\infty.
\end{equation}
Moreover, if
\begin{equation}\label{eq:rhoN-definition}
 \rho_N(U;I)
 :=
 \|P_{\le N}f(U)-f(P_{\le N}U)\|_
 {L_t^2W_x^{1,(2^*)'}(I)},
\end{equation}
then
\begin{equation}\label{eq:rhoN-zero}
 \rho_N(U;I)\to0
 \qquad\text{as }N\to\infty.
\end{equation}
\end{lemma}

\begin{proof}
The convolution kernel of \(P_{\le N}\) is a dilation of a fixed
Schwartz function.  Hence these operators are uniformly bounded on
\(L^q(\R^d)\), \(1\le q\le\infty\), and on
\(W^{1,q}(\R^d)\), \(1<q<\infty\), and converge strongly to the
identity on the corresponding spaces when the exponent is finite.
Applying these bounds to the four components of \(\cX^1(I)\) yields
\eqref{eq:PN-uniform-X}.

For the three components with finite time exponent, spatial strong
convergence and the uniform multiplier bounds imply, by dominated
convergence in time,
\[
 \|(P_{\le N}-1)U\|_{L_t^2W_x^{1,2^*}(I)}
 +\|(P_{\le N}-1)\nabla U\|_{L_t^{q_c}L_x^{r_c}(I)}
 +\|(P_{\le N}-1)U\|_{L_{t,x}^{p_c}(I)}
 \to0.
\]
For the remaining component, the set
\(
 \{U(t):t\in I\}
\)
is compact in \(H^1\), because \(U\in C(I;H^1)\) and \(I\) is
compact.  The uniformly bounded operators \(P_{\le N}-1\) converge
strongly to zero on \(H^1\), and therefore uniformly on this compact
set.  Thus
\[
 \sup_{t\in I}\|(P_{\le N}-1)U(t)\|_{H^1}\to0,
\]
which proves \eqref{eq:PN-strong-X}.

For the nonlinear truncation error, write
\begin{align}
 P_{\le N}f(U)-f(P_{\le N}U)
 ={}&(P_{\le N}-1)f(U)+f(U)-f(P_{\le N}U).
 \label{eq:rhoN-split}
\end{align}
By \eqref{eq:nonlinear-single},
\(
 f(U)\in L_t^2W_x^{1,(2^*)'}(I),
\)
so the first term on the right-hand side of \eqref{eq:rhoN-split}
tends to zero in this space.  Lemma~\ref{lem:nonlinear-lip} yields
\begin{align*}
 \|f(U)-f(P_{\le N}U)\|_{L_t^2W_x^{1,(2^*)'}(I)}
 \le
 C_d\bigl(
 \|U\|_{\cZ^1(I)}^p
 +\|P_{\le N}U\|_{\cZ^1(I)}^p
 \bigr)
 \|U-P_{\le N}U\|_{\cX^1(I)},
\end{align*}
and the right-hand side tends to zero by
\eqref{eq:PN-uniform-X}--\eqref{eq:PN-strong-X}.  This proves
\eqref{eq:rhoN-zero}.
\end{proof}

\subsection{Frequency truncation and residual estimates}
\label{subsec:sch-H1-limit}

Let \(v\) solve the energy-critical NLS
\begin{equation}\label{eq:nls}
 i\partial_tv-\Delta v+\mu f(v)=0,
 \qquad v(0)=v_0.
\end{equation}
The standard local \(H^1\) theory provides, on every compact interval
\(I\) inside the maximal \(H^1\) lifespan,
\begin{equation}\label{eq:nls-standard}
 v\in C(I;H^1)\cap\cX^1(I).
\end{equation}
See, for instance,
\cite{CazenaveWeissler1990,Tao2006}; this is the standard local
\(H^1\) and Strichartz theory for energy-critical NLS.

The qualitative inviscid limit does not require \(H^3\) regularity of
\(v\).  Directly inserting \(v\) into the approximating CGL equation
produces the term
\(
 (A_\eps+i)\Delta v,
\)
which need not belong to \(L_t^1H_x^1\) for an \(H^1\) solution.
We therefore use \(v_N=P_{\le N}v\) as an approximate solution, so that
the two derivatives in the diffusion coefficient error
\((A_\eps+i)\Delta v_N\) fall on the frequency cutoff.

\begin{proposition}
\label{prop:sch-truncated-residual}
Let \(I=[0,T]\) and
\(
 v\in C(I;H^1)\cap\cX^1(I)
\)
solve \eqref{eq:nls}.  For \(N\ge1\), set
\(
 v_N=P_{\le N}v.
\)
For each \(\eps\), define
\begin{equation}\label{eq:truncated-residual-definition}
 e_{\eps,N}
 :=
 \partial_tv_N-A_\eps\Delta v_N
 +\mu C_\eps f(v_N)+r_\eps v_N.
\end{equation}
Then, in the sense of distributions,
\begin{equation}\label{eq:truncated-residual-expansion}
 \begin{aligned}
 e_{\eps,N}
 ={}-(A_\eps+i)\Delta v_N
 +\mu(C_\eps+i)f(v_N)
 +i\mu\bigl(P_{\le N}f(v)-f(v_N)\bigr)
 +r_\eps v_N.
 \end{aligned}
\end{equation}
Consequently, \(e_{\eps,N}\in\cN^1(I)\), and
\begin{equation}\label{eq:truncated-residual-estimate}
 \begin{aligned}
 \|e_{\eps,N}\|_{\cN^1(I)}
 \le{}&
 C_{\chi}T N^2|A_\eps+i|
 \|v\|_{L_t^\infty H_x^1(I)}
 \\
 &+C_{d,\chi}|C_\eps+i|
 \|v\|_{\cZ^1(I)}^p\|v\|_{\cX^1(I)}
 \\
 &+C_{\chi}T r_\eps
 \|v\|_{L_t^\infty H_x^1(I)}
 +\rho_N(v;I).
 \end{aligned}
\end{equation}
Consequently, there exists a finite constant
\(B_{v,T,\chi}\), independent of \(N\) and \(\eps\), such that
\begin{equation}\label{eq:truncated-residual-short}
 \|e_{\eps,N}\|_{\cN^1(I)}
 \le
 B_{v,T,\chi}
 \bigl(
 N^2|A_\eps+i|+|C_\eps+i|+r_\eps
 \bigr)
 +\rho_N(v;I).
\end{equation}
\end{proposition}

\begin{proof}
Since an \(H^1\) mild solution of \eqref{eq:nls} satisfies the equation
in the sense of distributions, and since \(P_{\le N}\) is a bounded
spatial Fourier multiplier commuting with \(\partial_t\) and
\(\Delta\),
\[
 \partial_tv_N
 =-i\Delta v_N+i\mu P_{\le N}f(v)
\]
in distributions.  Substitution into
\eqref{eq:truncated-residual-definition} yields
\eqref{eq:truncated-residual-expansion}.

We estimate the four terms in
\eqref{eq:truncated-residual-expansion} according to the two summands
in the definition of \(\cN^1(I)\).  Since the Fourier support of
\(v_N\) is contained in \(\{|\xi|\le2N\}\),
\[
 \|\Delta v_N(t)\|_{H^1}
 \le C_\chi N^2\|v(t)\|_{H^1}.
\]
Hence
\begin{equation}\label{eq:truncated-linear-error}
 \|(A_\eps+i)\Delta v_N\|_{L_t^1H_x^1(I)}
 \le
 C_\chi T N^2|A_\eps+i|
 \|v\|_{L_t^\infty H_x^1(I)}.
\end{equation}
Similarly,
\begin{equation}\label{eq:truncated-damping-error}
 \|r_\eps v_N\|_{L_t^1H_x^1(I)}
 \le
 C_\chi T r_\eps
 \|v\|_{L_t^\infty H_x^1(I)}.
\end{equation}

For the nonlinear coefficient error, the uniform multiplier bounds,
\eqref{eq:nonlinear-single}, and
\eqref{eq:PN-uniform-X} yield
\begin{align}
 \|(C_\eps+i)f(v_N)\|_
 {L_t^2W_x^{1,(2^*)'}(I)}
 &\le
 C_d|C_\eps+i|
 \|v_N\|_{\cZ^1(I)}^p
 \|v_N\|_{\cX^1(I)}\\
 &\le
 C_{d,\chi}|C_\eps+i|
 \|v\|_{\cZ^1(I)}^p
 \|v\|_{\cX^1(I)}.
 \label{eq:truncated-nonlinear-coefficient}
\end{align}
Finally,
\[
 \|P_{\le N}f(v)-f(v_N)\|_
 {L_t^2W_x^{1,(2^*)'}(I)}
 =\rho_N(v;I).
\]
Combining the last four estimates and using the Banach-sum definition
of \(\cN^1(I)\) proves
\eqref{eq:truncated-residual-estimate}.  Enlarging the constant yields
\eqref{eq:truncated-residual-short}.

We also record explicitly the forward mild formulation needed below.
Set
\(
 S_{A_\eps,r_\eps}(t)
 :=e^{-r_\eps t}e^{tA_\eps\Delta},
 ~t\ge0.
\)
Since \(I=[0,T]\), the distributional identity
\eqref{eq:truncated-residual-definition} implies
\begin{equation}\label{eq:truncated-mild-identity}
 \begin{aligned}
 v_N(t)
 ={}S_{A_\eps,r_\eps}(t)v_N(0)
 +\int_0^t
 S_{A_\eps,r_\eps}(t-s)
 \bigl[-\mu C_\eps f(v_N(s))+e_{\eps,N}(s)\bigr]\,\mathrm ds,
 \qquad 0\le t\le T.
 \end{aligned}
\end{equation}
Indeed, set \(F=-\mu C_\eps f(v_N)+e_{\eps,N}\).
By \eqref{eq:PN-uniform-X}, \eqref{eq:nonlinear-single}, and
\eqref{eq:truncated-residual-estimate},
\[
 \|F\|_{\cN^1(I)}
 \le C_{d,\chi}|C_\eps|\|v\|_{\cZ^1(I)}^p\|v\|_{\cX^1(I)}
       +\|e_{\eps,N}\|_{\cN^1(I)}<\infty.
\]
Let \(w\) denote the right-hand side of
\eqref{eq:truncated-mild-identity}.  Proposition~\ref{prop:sch-linear}
gives
\[
 \|w\|_{\cX^1(I)}
 \le C\bigl(\|v_N(0)\|_{H^1}+\|F\|_{\cN^1(I)}\bigr).
\]
Approximation of \(F\) in the two summands of \(\cN^1(I)\) also gives
\(w\in C(I;H^1)\), with
\(\partial_tw-A_\eps\Delta w+r_\eps w=F\) in distributions.
Thus \(z=v_N-w\) belongs to \(C(I;H^1)\) and satisfies
\[
 \partial_tz-A_\eps\Delta z+r_\eps z=0,\qquad z(0)=0.
\]
For \(L\ge1\), put \(z_L=P_{\le L}z\).  The equation implies
\(z_L\in C^1(I;L^2)\), so taking the real \(L^2\) scalar product
with \(z_L\) and integrating yields
\[
 \|z_L(t)\|_2^2
 +2\operatorname{Re}A_\eps\int_0^t\|\nabla z_L(s)\|_2^2\,\mathrm ds
 +2r_\eps\int_0^t\|z_L(s)\|_2^2\,\mathrm ds=0.
\]
Since \(\operatorname{Re}A_\eps,r_\eps\ge0\), we have \(z_L=0\).
Letting \(L\to\infty\) gives \(z=0\), which proves
\eqref{eq:truncated-mild-identity}.  Hence \(v_N\) satisfies the mild
formulation used in Proposition~\ref{prop:common-stability}.
\end{proof}

\begin{remark}
\label{rem:sch-two-parameter}
Let \(K_*\) be the stability constant from
Proposition~\ref{prop:common-stability}, with the uniform linear estimate
of Proposition~\ref{prop:sch-linear} and the coefficient bounds fixed
as in the proof below.  Whenever the corresponding stability
smallness condition is satisfied, Proposition~\ref{prop:common-stability}
and the triangle inequality yield
\begin{align}
 \|v^\eps-v\|_{\cX^1(I)}
 \le{}&
 \|P_{\le N}v-v\|_{\cX^1(I)}
 \notag\\
 &+K_*\Bigl[
 \|v_0^\eps-P_{\le N}v_0\|_{H^1}
 +B_{v,T,\chi}
 \bigl(
 N^2|A_\eps+i|+|C_\eps+i|+r_\eps
 \bigr)
 +\rho_N(v;I)
 \Bigr].
 \label{eq:two-parameter-inviscid-bound}
\end{align}
This estimate makes the qualitative nature of the \(H^1\) argument
explicit.  The truncation terms tend to zero as \(N\to\infty\), but the
\(H^1\) theory provides no quantitative rate in \(N\), while the coefficient
error contains the factor \(N^2|A_\eps+i|\).  We therefore first choose \(N\)
large and then, with \(N\) fixed, make the coefficient and initial-data
errors small.
\end{remark}

\subsection{Qualitative inviscid limit}

We first prove the inviscid limit at the natural energy regularity.

\begin{proof}[Proof of Theorem~\ref{thm:intro-B}: qualitative convergence]
Fix
\(
 I=[0,T],~ 0<T<T_{\max}(v_0).
\)
By the standard local theory,
\(
 v\in C(I;H^1)\cap\cX^1(I).
\)
Let \(C_\chi\) be chosen so that
\[
 \sup_{N\ge1}\|P_{\le N}v\|_{\cX^1(I)}
 \le C_\chi\|v\|_{\cX^1(I)},
\]
and set
\(
 M=1+C_\chi\|v\|_{\cX^1(I)}.
\)
Write
\(
 A_\eps=a_\eps-ib_\eps.
\)
By the coefficient assumptions
\(
 |A_\eps|=1,~ \operatorname{Re}A_\eps\ge0,~
 A_\eps\to-i,~ C_\eps\to-i,~ r_\eps\to0,
\)
there exists \(\eps_0>0\) such
that, for \(0<\eps<\eps_0\),
\(
 a_\eps\ge0,~
 b_\eps\ge 1 / 2,~
 a_\eps^2+b_\eps^2=1,~
 |C_\eps|\le2,~
 0\le r_\eps\le1.
\)
Proposition~\ref{prop:sch-linear}, with \(b_0=1/2\) and \(r_0=1\),
supplies the uniform linear estimate required in
Proposition~\ref{prop:common-stability}.  Apply that proposition with
\(C_0=2\) and with the fixed reference bound \(M\), and denote the
corresponding constants by \(\varepsilon_*>0\) and \(K_*>0\).
Crucially, these constants are independent of \(N\).

Let \(\eta>0\).  By
Lemma~\ref{lem:frequency-truncation}, we may first choose
\(N=N(\eta)\ge1\) so large that
\begin{equation}\label{eq:choose-N-1}
 \|P_{\le N}v-v\|_{\cX^1(I)}
 <\frac{\eta}{2},
\end{equation}
and
\begin{equation}\label{eq:choose-N-2}
 \|(1-P_{\le N})v_0\|_{H^1}
 +\rho_N(v;I)
 <
 \min\left\{
 \frac{\varepsilon_*}{4},
 \frac{\eta}{4K_*}
 \right\}.
\end{equation}
This \(N\) is now fixed.

By the assumptions
\(
 \|v_0^\eps-v_0\|_{H^1}\to0,~
 A_\eps\to-i,~
 C_\eps\to-i,~
 r_\eps\to0,
\)
and by Proposition~\ref{prop:sch-truncated-residual}, we may then
choose \(0<\eps_\eta<\eps_0\) so that, for \(0<\eps<\eps_\eta\),
\begin{equation}\label{eq:choose-eps}
 \|v_0^\eps-v_0\|_{H^1}
 +B_{v,T,\chi}
 \bigl(
 N^2|A_\eps+i|+|C_\eps+i|+r_\eps
 \bigr)
 <
 \min\left\{
 \frac{\varepsilon_*}{4},
 \frac{\eta}{4K_*}
 \right\}.
\end{equation}

Set \(U=v_N=P_{\le N}v\).  Since \(U(0)=P_{\le N}v_0\),
\eqref{eq:truncated-residual-short},
\eqref{eq:choose-N-2}, and \eqref{eq:choose-eps} imply
\begin{align*}
 \|v_0^\eps-U(0)\|_{H^1}
 +\|e_{\eps,N}\|_{\cN^1(I)}
 &\le
 \|v_0^\eps-v_0\|_{H^1}
 +\|(1-P_{\le N})v_0\|_{H^1}
 \\
 &\quad+B_{v,T,\chi}
 \bigl(
 N^2|A_\eps+i|+|C_\eps+i|+r_\eps
 \bigr)
 +\rho_N(v;I)
 \\
 &<\varepsilon_*.
\end{align*}
Therefore Proposition~\ref{prop:common-stability} applies to the
approximate solution \(U=P_{\le N}v\).  It produces a unique exact
solution
\(
 v^\eps\in C(I;H^1)\cap\cX^1(I)
\)
of the normalized CGL equation and yields
\begin{align*}
 \|v^\eps-P_{\le N}v\|_{\cX^1(I)}
 \le
 K_*
 \bigl(
 \|v_0^\eps-P_{\le N}v_0\|_{H^1}
 +\|e_{\eps,N}\|_{\cN^1(I)}
 \bigr)
<\frac{\eta}{2}.
\end{align*}
Combining this estimate with \eqref{eq:choose-N-1}, we obtain
\[
 \|v^\eps-v\|_{\cX^1(I)}
 \le
 \|v^\eps-P_{\le N}v\|_{\cX^1(I)}
 +\|P_{\le N}v-v\|_{\cX^1(I)}
 <\eta.
\]
Taking, for instance, \(\eta=1\) in the preceding construction first
shows that the exact CGL solution exists on \(I\) for every sufficiently
small \(\eps\).  For an arbitrary \(\eta>0\), repeating the construction
may use a different frequency cutoff \(N(\eta)\), but
Proposition~\ref{prop:common-stability} ensures uniqueness in
\(C(I;H^1)\cap\cX^1(I)\); hence the solution obtained using the new
approximation is the same solution.  The preceding
estimate therefore shows that, for every \(\eta>0\),
\(
 \|v^\eps-v\|_{\cX^1(I)}<\eta
\)
for all sufficiently small \(\eps\).  Thus
\(
 v^\eps\to v
 ~\text{in }\cX^1(I).
\)
\end{proof}

\subsection{A linear obstruction to algebraic convergence rates}
\label{subsec:linear-rate-obstruction}

\begin{proposition}\label{prop:linear-rate-obstruction}
Let \(3\le d\le6\), fix \(T>0\), and set
\[
 A_\eps=\eps-i\sqrt{1-\eps^2},\qquad 0<\eps<1.
\]
This is an admissible normalized inviscid path, with
\(|A_\eps+i|\sim\eps\).  For every \(g\in H^1(\R^d)\),
\begin{equation}\label{eq:linear-inviscid-strong}
 \sup_{0\le t\le T}
 \|e^{tA_\eps\Delta}g-e^{-it\Delta}g\|_{H^1}\to0
 \qquad\text{as }\eps\to0.
\end{equation}
Nevertheless, there exist a fixed \(g\in H^1(\R^d)\) and
\(\eps_n=e^{-n^2}\) such that, for every \(\delta>0\),
\begin{equation}\label{eq:linear-rate-obstruction}
 \frac{\|e^{TA_{\eps_n}\Delta}g-e^{-iT\Delta}g\|_{H^1}}
 {\eps_n^\delta}\to\infty
 \qquad\text{as }n\to\infty.
\end{equation}
Thus no algebraic rate depending only on \(H^1\) regularity can in general
be expected, even at the linear level.
\end{proposition}

\begin{proof}
Since \(|A_\eps|=1\) and
\(1-\sqrt{1-\eps^2}=\eps^2/(1+\sqrt{1-\eps^2})\), we have
\(|A_\eps+i|\sim\eps\).  After factoring out the Schr\"odinger multiplier
\(e^{it|\xi|^2}\), the difference multiplier is
\[
 m_\eps(t,\xi)
 =e^{-\eps t|\xi|^2}
  e^{-i(1-\sqrt{1-\eps^2})t|\xi|^2}-1.
\]
For each \(\xi\), \(\sup_{0\le t\le T}|m_\eps(t,\xi)|\to0\), and this
supremum is bounded by \(2\).  By Plancherel's identity and dominated convergence,
we obtain \eqref{eq:linear-inviscid-strong}.

Set \(\eps_n=e^{-n^2}\) and choose
\(\widehat g_n\in C_c^\infty(\R^d)\) supported in
\[
 E_n=\{\xi:1<T\eps_n|\xi|^2<2\},
 \qquad \|g_n\|_{H^1}=1.
\]
The annuli are disjoint, since
\(\eps_n/\eps_{n+1}=e^{2n+1}>2\).  Hence
\(g=\sum_{n\ge1}n^{-1}g_n\) belongs to \(H^1\), with
\(\|g\|_{H^1}^2=\sum_{n\ge1}n^{-2}\).
On \(E_n\), the reverse triangle inequality gives
\[
 |m_{\eps_n}(T,\xi)|
 \ge1-e^{-\eps_nT|\xi|^2}\ge1-e^{-1}.
\]
By disjointness and Plancherel's identity,
\[
 \|e^{TA_{\eps_n}\Delta}g-e^{-iT\Delta}g\|_{H^1}
 \ge\frac{1-e^{-1}}n.
\]
Dividing by \(\eps_n^\delta=e^{-\delta n^2}\) proves
\eqref{eq:linear-rate-obstruction}.
\end{proof}

\subsection{Higher regularity and quantitative estimate}

\begin{lemma}
\label{lem:nls-H3-persistence}
Let \(3\le d\le6\), let \(\mu\in\{-1,1\}\), and let \(v\) be the
maximal \(H^1\) solution of \eqref{eq:nls}.  If
\(v_0\in H^3(\R^d)\), then
\(
 v\in C([0,T];H^3(\R^d))
\)
for every \(0<T<T_{\max}(v_0)\).
Moreover, on any compact interval \(I=[t_0,t_1]\) in its lifespan,
\(\|v\|_{\cX^1(I)}\le M\) implies
\begin{equation}\label{eq:nls-H3-tame}
 \|v\|_{L_t^\infty H_x^3(I)}
 \le C(d,M,I)\|v(t_0)\|_{H^3}.
\end{equation}
The constant is independent of the \(H^3\) norm of the initial data.
\end{lemma}

\begin{proof}
We first work on an interval \(I=[t_0,t_1]\) on which
\(v\in C(I;H^3)\).  Let \(K_d\) be
the constant in the NLS linear estimate.  Choose \(\eta>0\) so that
\(K_dC_d\eta^p\le1/2\), where \(C_d\) bounds the nonlinear estimates
below.  Partition \(I\) into intervals \(J_j=[\tau_j,\tau_{j+1}]\) on
which \(\|v\|_{\cZ^1(J_j)}\le\eta\).  Since \(q_c=p_c<\infty\),
the number of intervals is bounded in terms of \(d\) and \(M\).

For \(d=3,4\), the nonlinearity is a polynomial.  For \(1\le k\le3\),
by the product rule and Gagliardo--Nirenberg interpolation in space,
\[
 \|D^kf(v)\|_{L_x^{(2^*)'}}
 \le C_d\|v\|_{L_x^{p_c}}^p\|D^kv\|_{L_x^{r_c}}.
\]
Indeed, for \(0<\ell<k\),
\[
 \|D^\ell v\|_{L_x^{a_\ell}}
 \le C_d\|v\|_{L_x^{p_c}}^{1-\ell/k}
          \|D^kv\|_{L_x^{r_c}}^{\ell/k},
 \qquad
 \frac1{a_\ell}=\frac{1-\ell/k}{p_c}+\frac{\ell/k}{r_c}.
\]
Each product in \(D^kf(v)\) contains \(p+1\) factors with a total
of \(k\) derivatives.  Applying the preceding inequality and
\eqref{eq:critical-identities} yields the asserted bound.
The estimate for \(k=0\) follows from H\"older's inequality.
Set
\[
 Y_3(J):=\sum_{|\alpha|\le2}\|\partial^\alpha v\|_{\cX^1(J)}.
\]
H\"older's inequality in time, followed by the linear estimate for
\(\partial^\alpha v\), yields
\[
 Y_3(J_j)
 \le K_d\|v(\tau_j)\|_{H^3}
   +K_dC_d\|v\|_{L_{t,x}^{p_c}(J_j)}^pY_3(J_j).
\]
Absorbing the last term yields \(Y_3(J_j)\le2K_d\|v(\tau_j)\|_{H^3}\).
Iteration over the partition proves \eqref{eq:nls-H3-tame}.

For \(d=5,6\), we differentiate in time, as in
Huang and Wang~\cite[Proposition~3.6 and Remark~3.7]{HuangWang2008}.
The real derivative of \(f\) satisfies
\[
 |Df(v)y|\le C_d|v|^p|y|,\qquad
 |\nabla(Df(v)y)|
 \le C_d\bigl(|v|^p|\nabla y|
       +|v|^{p-1}|\nabla v||y|\bigr).
\]
Using \eqref{eq:critical-identities} and \eqref{eq:aux-interp}, we obtain
\begin{align*}
 \|Df(v)y\|_{L_t^2W_x^{1,(2^*)'}(J)}
 &\le C_d\|v\|_{L_{t,x}^{p_c}(J)}^p
       \bigl(\|y\|_{L_t^{q_c}L_x^{r_c}(J)}
             +\|\nabla y\|_{L_t^{q_c}L_x^{r_c}(J)}\bigr)\\
 &\quad+C_d\|v\|_{L_{t,x}^{p_c}(J)}^{p-1}
       \|\nabla v\|_{L_t^{q_c}L_x^{r_c}(J)}
       \|y\|_{L_{t,x}^{p_c}(J)}\\
 &\le C_d\|v\|_{\cZ^1(J)}^p\|y\|_{\cX^1(J)}.
\end{align*}
When \(p=1\), the factor with exponent \(p-1\) is omitted.
For \(y=\partial_jv\) or \(y=\partial_tv\), differentiation gives
\[
 i\partial_ty-\Delta y+\mu Df(v)y=0.
\]
Hence, on each \(J_j\),
\[
 \|y\|_{\cX^1(J_j)}
 \le K_d\|y(\tau_j)\|_{H^1}
      +K_dC_d\eta^p\|y\|_{\cX^1(J_j)}
 \le K_d\|y(\tau_j)\|_{H^1}
      +\tfrac12\|y\|_{\cX^1(J_j)}.
\]
Iteration, together with the same estimate for the undifferentiated
equation, yields
\[
 \|v\|_{L_t^\infty H_x^k(I)}
 \le C(d,M,I)\|v(t_0)\|_{H^k}\quad(k=1,2),\qquad
 \|\partial_tv\|_{\cX^1(I)}
 \le C(d,M,I)\|\partial_tv(t_0)\|_{H^1}.
\]

Put \(b=2d/(d-4)\).  By the Sobolev embedding and H\"older's inequality,
\begin{align*}
 \|f(g)\|_{H^1}
 &\le C_d\|g\|_{2^*}^{p-1}\|g\|_b
              \bigl(\|g\|_{2^*}+\|\nabla g\|_{2^*}\bigr)\\
 &\le C_d\|g\|_{H^1}^{p-1}\|g\|_{H^2}^2,\\
 \|g\|_{H^2}^2&\le\|g\|_{H^1}\|g\|_{H^3}.
\end{align*}
Here \(1/2=(p-1)/2^*+1/b+1/2^*\).
Writing \(g=v(t_0)\) and using \(\|g\|_{H^1}\le M\), we obtain
\begin{align*}
 \|f(v)\|_{L_t^\infty H_x^1(I)}
 &\le C(d,M,I)\|g\|_{H^1}^{p-1}\|g\|_{H^2}^2
 \le C(d,M,I)\|g\|_{H^3},\\
 \|\partial_tv(t_0)\|_{H^1}
 &\le \|\Delta g\|_{H^1}+\|f(g)\|_{H^1}
 \le C(d,M)\|g\|_{H^3}.
\end{align*}
Since \(\Delta v=i\partial_tv+\mu f(v)\),
\[
 \|v\|_{L_t^\infty H_x^3(I)}
 \le C_d\bigl(\|v\|_{L_t^\infty H_x^1(I)}
       +\|\partial_tv\|_{L_t^\infty H_x^1(I)}
       +\|f(v)\|_{L_t^\infty H_x^1(I)}\bigr)
 \le C(d,M,I)\|g\|_{H^3}.
\]
This proves \eqref{eq:nls-H3-tame}.

To complete the persistence argument, let \(t_*\le t_1\) be a possible
endpoint of the \(H^3\) lifespan.  The local theory in
\cite[Proposition~3.6 and Remark~3.7]{HuangWang2008} justifies the
preceding estimates on \([t_0,t_*)\).  If
\(i\partial_ty-\Delta y=F\), the Duhamel formula and the dual
Strichartz estimate give, for \(t_0\le s<t<t_*\),
\[
 \|e^{it\Delta}y(t)-e^{is\Delta}y(s)\|_{H^1}
 \le C_d\|F\|_{L_t^2W_x^{1,(2^*)'}([s,t])}.
\]
When \(F\in L_t^2W_x^{1,(2^*)'}([t_0,t_*))\), the right-hand side
tends to zero as \(s,t\uparrow t_*\).  Thus \(y(t)\) has an \(H^1\)
limit.

For \(d=3,4\), apply this estimate to
\(y=\partial^\alpha v\), \(F=-\mu\partial^\alpha f(v)\),
\(|\alpha|\le2\).  The bounds above imply that \(v(t)\) converges
in \(H^3\).  For \(d=5,6\), take
\(y=\partial_jv,\partial_tv\) and \(F=-\mu Df(v)y\).
Then \(v(t)\) converges in \(H^2\), and \(\partial_tv(t)\)
converges in \(H^1\).  Since \(f:H^2\to H^1\) is continuous,
\begin{align*}
 \|v(t)-v(s)\|_{H^3}
 \le C_d\bigl(&\|v(t)-v(s)\|_{H^1}
 +\|\partial_tv(t)-\partial_tv(s)\|_{H^1}\\
 &+\|f(v(t))-f(v(s))\|_{H^1}\bigr)\to0
 .
\end{align*}
In both cases the \(H^3\) limit provides initial data for continuation
at \(t_*\).  Hence \(v\in C(I;H^3)\), and
\eqref{eq:nls-H3-tame} holds on \(I\).
\end{proof}

We next prove the quantitative refinement for \(H^{1+s}\) data,
\(0<s\le2\), using NLS solutions with truncated initial data.
The following residual estimate will be applied to these solutions.

\begin{proposition}
\label{prop:sch-residual}
Let \(I=[0,T]\) and assume
\(
 v\in C(I;H^3)\cap\cX^1(I).
\)
Define
\[
 e_\eps
 =-(A_\eps+i)\Delta v
 +\mu(C_\eps+i)f(v)
 +r_\eps v.
\]
Then
\begin{align}
 \|e_\eps\|_{\cN^1(I)}
 \le{}
 T|A_\eps+i|\|v\|_{L_t^\infty H_x^3(I)}
 +C_d|C_\eps+i|
 \|v\|_{\cZ^1(I)}^p\|v\|_{\cX^1(I)}
 +Tr_\eps\|v\|_{L_t^\infty H_x^1(I)}.
 \label{eq:sch-residual}
\end{align}
\end{proposition}

\begin{proof}
Split
\[
 e_{\eps,1}=-(A_\eps+i)\Delta v+r_\eps v,
 \qquad
 e_{\eps,2}=\mu(C_\eps+i)f(v).
\]
Since
\(
 \|\Delta v(t)\|_{H^1}\le\|v(t)\|_{H^3},
\)
we have
\[
 \|e_{\eps,1}\|_{L_t^1H_x^1(I)}
 \le
 T|A_\eps+i|\|v\|_{L_t^\infty H_x^3(I)}
 +Tr_\eps\|v\|_{L_t^\infty H_x^1(I)}.
\]
Equation \eqref{eq:nonlinear-single} implies
\[
 \|e_{\eps,2}\|_{L_t^2W_x^{1,(2^*)'}(I)}
 \le
 C_d|C_\eps+i|
 \|v\|_{\cZ^1(I)}^p
 \|v\|_{\cX^1(I)}.
\]
The Banach-sum definition of \(\cN^1(I)\) yields
\eqref{eq:sch-residual}.
\end{proof}

\begin{proof}[Proof of Theorem~\ref{thm:intro-B}: quantitative convergence]
Fix \(I=[0,T]\), \(0<s\le2\), and \(v_0\in H^{1+s}\).
For dyadic \(N\ge1\), let \(\widetilde v_N\) be the NLS solution with
\[
 \widetilde v_N(0)=P_{\le N}v_0.
\]
By the support properties of \(\chi\) and Plancherel's identity,
\[
 \|(1-P_{\le N})v_0\|_{H^1}
 \lesssim N^{-s}\|v_0\|_{H^{1+s}},\qquad
 \|P_{\le N}v_0\|_{H^3}
 \lesssim N^{2-s}\|v_0\|_{H^{1+s}}.
\]
Indeed, the first multiplier vanishes for \(|\xi|\le N\), while the
second is supported in \(|\xi|\le2N\); on these sets, respectively,
\[
 (1+|\xi|^2)^{-s/2}\lesssim N^{-s},\qquad
 (1+|\xi|^2)^{(2-s)/2}\lesssim N^{2-s}.
\]
Apply Proposition~\ref{prop:common-stability} with \(A=C=-i\), \(r=0\),
reference solution \(v\), and zero residual.  Choose \(N_0\) so that
the first cutoff bound is below its smallness threshold for \(N\ge N_0\).
The proposition supplies \(\widetilde v_N\) on \(I\), with
\begin{equation}\label{eq:sch-exact-approximation}
 \|\widetilde v_N-v\|_{\cX^1(I)}
 \le C_{v,T,s}N^{-s},\qquad
 \sup_{N\ge N_0}\|\widetilde v_N\|_{\cX^1(I)}<\infty.
\end{equation}
Here and below the constants may depend on \(\|v_0\|_{H^{1+s}}\),
but not on \(N\) or \(\eps\).
Fix a bound \(M_1\) for the supremum in
\eqref{eq:sch-exact-approximation}.  Lemma~\ref{lem:nls-H3-persistence}
then yields
\[
 \|\widetilde v_N\|_{L_t^\infty H_x^3(I)}
 \le C(d,M_1,I)\|P_{\le N}v_0\|_{H^3}
 \le C_{v,T,s}N^{2-s}.
\]

The NLS equation for \(\widetilde v_N\) gives
\[
 \partial_t\widetilde v_N-A_\eps\Delta\widetilde v_N
 +\mu C_\eps f(\widetilde v_N)+r_\eps\widetilde v_N=e_{\eps,N},
\]
where
\[
 e_{\eps,N}=-(A_\eps+i)\Delta\widetilde v_N
 +\mu(C_\eps+i)f(\widetilde v_N)+r_\eps\widetilde v_N.
\]
Proposition~\ref{prop:sch-residual} implies
\begin{align*}
 \|e_{\eps,N}\|_{\cN^1(I)}
 &\le T|A_\eps+i|\|\widetilde v_N\|_{L_t^\infty H_x^3(I)}
       +C_d|C_\eps+i|M_1^{p+1}+Tr_\eps M_1\\
 &\le C_{v,T,s}
       \bigl(|A_\eps+i|N^{2-s}+|C_\eps+i|+r_\eps\bigr).
\end{align*}
Also,
\[
 \|u_0^\eps-P_{\le N}v_0\|_{H^1}
 \le\|u_0^\eps-v_0\|_{H^1}+C_{v,T,s}N^{-s}.
\]
For sufficiently small \(\eps\), Proposition~\ref{prop:sch-linear}
provides the linear bound in Proposition~\ref{prop:common-stability}
with a constant independent of \(\eps\), and \(|C_\eps|\le2\).
Apply the latter proposition with reference bound \(M_1\) and \(C_0=2\).
Its constants \(\varepsilon_*\) and \(K_*\) are independent of \(N\)
and \(\eps\).  Thus, whenever
\[
 \|u_0^\eps-v_0\|_{H^1}
 +C_{v,T,s}\bigl(N^{-s}+|A_\eps+i|N^{2-s}
                       +|C_\eps+i|+r_\eps\bigr)
 \le\varepsilon_*,
\]
it gives
\[
 \|u^\eps-\widetilde v_N\|_{\cX^1(I)}
 \le K_*\bigl(\|u_0^\eps-P_{\le N}v_0\|_{H^1}
                   +\|e_{\eps,N}\|_{\cN^1(I)}\bigr).
\]
Combining this with \eqref{eq:sch-exact-approximation}, we obtain
\begin{equation}\label{eq:sch-fractional-cutoff}
 \|u^\eps-v\|_{\cX^1(I)}
 \le C_{v,T,s}\bigl(
 \|u_0^\eps-v_0\|_{H^1}+N^{-s}
 +|A_\eps+i|N^{2-s}+|C_\eps+i|+r_\eps\bigr).
\end{equation}

Write \(\delta_\eps=|A_\eps+i|\).
If \(\delta_\eps>0\), choose a dyadic \(N_\eps\) such that
\(
 \delta_\eps^{-1/2}\le N_\eps<2\delta_\eps^{-1/2}.
\)
For sufficiently small \(\eps\), \(N_\eps\ge N_0\), and
\(
 N_\eps^{-s}\le\delta_\eps^{s/2},~
 \delta_\eps N_\eps^{2-s}
 \le2^{2-s}\delta_\eps^{s/2}.
\)
Since
\[
 \|u_0^\eps-v_0\|_{H^1}
 +\delta_\eps^{s/2}+|C_\eps+i|+r_\eps\to0,
\]
the smallness condition holds, and
\eqref{eq:sch-fractional-cutoff} proves
\eqref{eq:intro-schrodinger-limit}.

If \(\delta_\eps=0\), use \(v\) as the reference solution.  Its residual
is \(e_\eps=\mu(C_\eps+i)f(v)+r_\eps v\), and
\[
 \|e_\eps\|_{\cN^1(I)}
 \le C_d|C_\eps+i|\|v\|_{\cZ^1(I)}^p\|v\|_{\cX^1(I)}
       +Tr_\eps\|v\|_{L_t^\infty H_x^1(I)}.
\]
Proposition~\ref{prop:common-stability} gives
\[
 \|u^\eps-v\|_{\cX^1(I)}
 \le C_{v,T}\bigl(\|u_0^\eps-v_0\|_{H^1}
                       +|C_\eps+i|+r_\eps\bigr).
\]
Taking the larger of the constants in the two cases proves the claimed
estimate for all sufficiently small \(\eps\).  At \(s=2\), the
diffusion-coefficient term is \(|A_\eps+i|\).
\end{proof}

 \section{The zero-dispersion limit}\label{sec:heat}

We finally consider the zero-dispersion limit.  The real part of the
coefficient supplies uniform parabolic smoothing, which controls the
diffusion-coefficient error \((A_\eps-1)\Delta h\) directly at the energy
regularity.

\subsection{Uniform parabolic estimates}

We use the following energy estimate, whose constant depends only on a
positive lower bound for \(\operatorname{Re}A\).

\begin{proposition}\label{prop:heat-linear}
Let $I=[t_0,t_1]$, $\operatorname{Re}A\ge a_0>0$, and $r\ge0$.
For $z_0\in H^1(\R^d)$ and $F\in\cN^1(I)$, the mild solution of
\[
 \partial_tz-A\Delta z+rz=F,
 \qquad z(t_0)=z_0\in H^1(\R^d),
\]
satisfies
\begin{equation}\label{eq:heat-linear}
 \|z\|_{\cX^1(I)}+\|\nabla z\|_{L_t^2H_x^1(I)}
 \le C(d,a_0)
 \bigl(\|z_0\|_{H^1}+\|F\|_{\cN^1(I)}\bigr).
\end{equation}
More generally, if
\[
 \partial_tz-A\Delta z+rz=\beta\Delta U+F,\qquad z(t_0)=z_0,
\]
where \(\beta\in\C\) and \(\nabla U\in L_t^2H_x^1(I)\), then
\begin{equation}\label{eq:heat-structured-linear}
 \|z\|_{\cX^1(I)}+\|\nabla z\|_{L_t^2H_x^1(I)}
 \le C(d,a_0)\bigl(\|z_0\|_{H^1}
 +|\beta|\|\nabla U\|_{L_t^2H_x^1(I)}
 +\|F\|_{\cN^1(I)}\bigr).
\end{equation}
The constants are independent of \(A,r,I,\beta,U,z_0,F\).
\end{proposition}

\begin{proof}
We first take $r=0$ and write $F=F_1+F_2$ according to the two summands
of $\cN^1(I)$.  Apply $D^k$,
$k=0,1$, take the real $L^2$ scalar product with $D^kz$, and integrate by
parts. The imaginary part of $A$ makes no contribution to the energy and the
lower bound on its real part yields
\[
 \frac12\frac{\mathrm d}{\mathrm dt}\|D^kz\|_2^2
 +a_0\|\nabla D^kz\|_2^2
 \le |\langle D^kF_1,D^kz\rangle|
 +|\langle D^kF_2,D^kz\rangle|.
\]
Integrating the $F_1$ term in time gives
\[
 \int_{t_0}^t|\langle D^kF_1,D^kz\rangle|\,\mathrm d\tau
 \le \|D^kF_1\|_{L_t^1L_x^2([t_0,t])}
 \|D^kz\|_{L_t^\infty L_x^2([t_0,t])}.
\]
For $F_2$, by the Sobolev and Young inequalities,
\[
 |\langle D^kF_2,D^kz\rangle|
 \lesssim_d \|D^kF_2\|_{L_x^{(2^*)'}}
 \|\nabla D^kz\|_2
 \le \frac{a_0}{2}\|\nabla D^kz\|_2^2
 +C(d,a_0)\|D^kF_2\|_{L_x^{(2^*)'}}^2.
\]
After integration in time, we use
$XY\le X^2/4+Y^2$ to absorb the factor containing the time supremum.
Consequently,
\begin{align*}
 \|D^kz\|_{L_t^\infty L_x^2(I)}^2
 +a_0\|\nabla D^kz\|_{L_{t,x}^2(I)}^2
 \le C(d,a_0)\bigl(
 \|D^kz_0\|_2^2+
 \|D^kF_1\|_{L_t^1L_x^2(I)}^2+
 \|D^kF_2\|_{L_t^2L_x^{(2^*)'}(I)}^2\bigr).
\end{align*}
By Sobolev's inequality,
\begin{equation}\label{eq:heat-energy}
 \|D^kz\|_{L_t^\infty L_x^2}
 +\|D^kz\|_{L_t^2L_x^{2^*}}
 \le C(d,a_0)\bigl(
 \|D^kz_0\|_2+\|D^kF_1\|_{L_t^1L_x^2}
 +\|D^kF_2\|_{L_t^2L_x^{(2^*)'}}\bigr).
\end{equation}
Using \eqref{eq:aux-interp} with $w=\nabla z$ yields the
$L_t^{q_c}L_x^{r_c}$ gradient bound.  Applying \eqref{eq:sob-pc} then
yields the $L_{t,x}^{p_c}$ bound.  Summing the estimates for $k=0,1$ and
taking the infimum over all decompositions of $F$ proves
\eqref{eq:heat-linear} when $r=0$.

For general $r\ge0$, the same energy argument applies directly.  Indeed,
the term $rz$ contributes
\(
 r\|D^kz\|_2^2
\)
to the left-hand side of the energy identity and is therefore
nonnegative.  Dropping this term yields exactly the estimate obtained
when $r=0$.

For the term \(\beta\Delta U\), integration by parts gives, for \(k=0,1\),
\[
 \langle D^k\Delta U,D^kz\rangle
 =-\langle\nabla D^kU,\nabla D^kz\rangle.
\]
By Young's inequality,
\[
 |\beta|\,|\langle\nabla D^kU,\nabla D^kz\rangle|
 \le \frac{a_0}{4}\|\nabla D^kz\|_2^2
 +C(a_0)|\beta|^2\|\nabla D^kU\|_2^2.
\]
After integration in time and summation over \(k=0,1\), the preceding
energy estimates yield
\begin{align*}
 \|z\|_{L_t^\infty H_x^1(I)}^2
   +\|\nabla z\|_{L_t^2H_x^1(I)}^2
 \le C(d,a_0)\bigl(
       \|z_0\|_{H^1}^2+\|F_1\|_{L_t^1H_x^1(I)}^2
       +\|F_2\|_{L_t^2W_x^{1,(2^*)'}(I)}^2
       +|\beta|^2\|\nabla U\|_{L_t^2H_x^1(I)}^2\bigr).
\end{align*}
The spatial Sobolev embedding controls
\(\|z\|_{L_t^2W_x^{1,2^*}(I)}\).  Equations
\eqref{eq:aux-interp} and \eqref{eq:sob-pc} control the other two
components of \(\cX^1(I)\).  Taking square roots and then the infimum
over \(F=F_1+F_2\) proves \eqref{eq:heat-structured-linear} for smooth data.

To pass to the stated regularity, put \(G=\nabla U\), so that
\(\Delta U=\operatorname{div}G\).  Choose smooth approximations with
\[
 z_{0,n}\to z_0\ \text{in }H^1,\quad
 F_{1,n}\to F_1\ \text{in }L_t^1H_x^1,\quad
 F_{2,n}\to F_2\ \text{in }L_t^2W_x^{1,(2^*)'},\quad
 G_n\to G\ \text{in }L_t^2H_x^1.
\]
For \(S(t)=e^{-rt}e^{tA\Delta}\), define
\[
 z_n(t)=S(t-t_0)z_{0,n}
 +\int_{t_0}^tS(t-s)
       \bigl(\beta\operatorname{div}G_n+F_{1,n}+F_{2,n}\bigr)(s)
       \,\mathrm ds.
\]
The same integration by parts applies to \(\operatorname{div}G_n\).
Applying the estimate to \(z_n-z_m\), we obtain
\begin{align*}
 \|z_n-z_m\|_{\cX^1(I)}
   +\|\nabla(z_n-z_m)\|_{L_t^2H_x^1(I)}
 &\le C(d,a_0)\bigl(
       \|z_{0,n}-z_{0,m}\|_{H^1}
       +\|F_{1,n}-F_{1,m}\|_{L_t^1H_x^1(I)}\\
 &\quad +\|F_{2,n}-F_{2,m}\|_{L_t^2W_x^{1,(2^*)'}(I)}
       +|\beta|\|G_n-G_m\|_{L_t^2H_x^1(I)}\bigr)\to 0.
\end{align*}
Since each \(z_n\) is continuous
in \(H^1\), the \(L_t^\infty H_x^1\) bound gives convergence in
\(C(I;H^1)\), as well as in the other displayed norms.
Passing to the limit in the equation against test functions gives
\[
 \partial_tz-A\Delta z+rz=\beta\operatorname{div}G+F,\qquad
 z(t_0)=z_0.
\]
If \(\widetilde z_n\) is obtained from another smooth approximation of the
same data, the preceding estimate applied to \(z_n-\widetilde z_n\) gives
\(\|z_n-\widetilde z_n\|_{\cX^1(I)}\to0\).  Thus the limit does not
depend on the chosen approximations.
Taking \(z_0=F=0\) defines the Duhamel integral for
\(\beta\Delta U\) in \(C(I;H^1)\).
Finally, a difference of two solutions with the same data solves the
homogeneous equation with zero initial value.  Spatial regularization
and the energy inequality give
\[
 \|z(t)\|_2^2+2a_0\int_{t_0}^t\|\nabla z(s)\|_2^2\,\mathrm ds\le0,
\]
so the difference vanishes.
\end{proof}

\subsection{Quantitative zero-dispersion limit at \texorpdfstring{\(H^1\)}{H1} regularity}
\label{subsec:heat-H1-limit}

Let \(h\) be the limiting heat solution in
Theorem~\ref{thm:intro-zero-dispersion}, and fix
\(I=[0,T]\) with \(0<T<T_{\max}(h_0)\).
Since \(h\in C(I;H^1)\cap\cX^1(I)\),
\eqref{eq:nonlinear-single} implies
\(f(h)\in L_t^2W_x^{1,(2^*)'}(I)\).
Applying \eqref{eq:heat-linear} to
\(\partial_th-\Delta h=-\mu f(h)\) yields
\begin{equation}\label{eq:heat-H1-smoothing}
 \|\nabla h\|_{L_t^2H_x^1(I)}
 \le C_d\bigl(\|h_0\|_{H^1}
 +\|h\|_{\cZ^1(I)}^p\|h\|_{\cX^1(I)}\bigr)<\infty.
\end{equation}

\begin{proof}[Proof of Theorem~\ref{thm:intro-zero-dispersion}]
For sufficiently small \(\eps\), we have
\(\operatorname{Re}A_\eps\ge1/2\), \(|C_\eps|\le2\), and
\(0\le r_\eps\le1\).  If \(h\) itself is used as an approximate solution,
the diffusion-coefficient error is \((A_\eps-1)\Delta h\).  Estimating this
term directly in \(\cN^1\) would require one additional spatial derivative
of \(h\), whereas the uniform lower bound for \(\operatorname{Re}A_\eps\)
allows it to be treated by Proposition~\ref{prop:heat-linear}.

Let \(q^\eps\) solve
\[
 \partial_tq^\eps-A_\eps\Delta q^\eps+r_\eps q^\eps
 =(A_\eps-1)\Delta h,\qquad q^\eps(0)=0.
\]
Applying Proposition~\ref{prop:heat-linear} with
\(\beta=A_\eps-1\), \(U=h\), and zero initial data and \(F=0\),
and using \eqref{eq:heat-H1-smoothing}, we obtain
\[
 \|q^\eps\|_{\cX^1(I)}
 \le C_d|A_\eps-1|\|\nabla h\|_{L_t^2H_x^1(I)}
 \le C_{h,T}|A_\eps-1|.
\]
Set \(H^\eps=h+q^\eps\).  Adding the equations for \(h\) and \(q^\eps\)
gives
\[
 \partial_tH^\eps-A_\eps\Delta H^\eps
 +\mu C_\eps f(H^\eps)+r_\eps H^\eps=e^\eps,
 \qquad H^\eps(0)=h_0,
\]
where
\[
 e^\eps=\mu C_\eps\bigl(f(h+q^\eps)-f(h)\bigr)
 +\mu(C_\eps-1)f(h)+r_\eps h.
\]
Moreover, \(H^\eps\in C(I;H^1)\cap\cX^1(I)\) and
\[
 \|H^\eps\|_{\cX^1(I)}
 \le\|h\|_{\cX^1(I)}+C_{h,T}|A_\eps-1|.
\]
Thus these norms are bounded independently of \(\eps\).
Lemma~\ref{lem:nonlinear-lip} implies
\[
 \|f(h+q^\eps)-f(h)\|_{L_t^2W_x^{1,(2^*)'}(I)}
 \le C_{h,T}\|q^\eps\|_{\cX^1(I)}.
\]
Together with \eqref{eq:nonlinear-single} and
\(\|h\|_{L_t^1H_x^1(I)}\le T\|h\|_{L_t^\infty H_x^1(I)}\), this yields
\[
 \|e^\eps\|_{\cN^1(I)}
 \le C_{h,T}\bigl(|A_\eps-1|+|C_\eps-1|+r_\eps\bigr).
\]

By \eqref{eq:heat-linear}, the linear bound required in
Proposition~\ref{prop:common-stability} holds on every subinterval of
\(I\), uniformly in \(\eps\).  Since
\(\|u_0^\eps-h_0\|_{H^1}+\|e^\eps\|_{\cN^1(I)}\to0\),
that proposition applies to \(H^\eps\) for all sufficiently small
\(\eps\).  It gives a unique solution
\(u^\eps\in C(I;H^1)\cap\cX^1(I)\) satisfying
\[
 \|u^\eps-H^\eps\|_{\cX^1(I)}
 \le C_{h,T}\bigl(\|u_0^\eps-h_0\|_{H^1}
 +|A_\eps-1|+|C_\eps-1|+r_\eps\bigr).
\]
Finally, \(u^\eps-h=(u^\eps-H^\eps)+q^\eps\), so
\[
 \|u^\eps-h\|_{\cX^1(I)}
 \le K_{h,T}\bigl(\|u_0^\eps-h_0\|_{H^1}
 +|A_\eps-1|+|C_\eps-1|+r_\eps\bigr).
\]
This proves \eqref{eq:intro-heat-limit}, and
\eqref{eq:intro-heat-limit-H1} follows.
\end{proof}
\medskip

\bibliographystyle{plain}

\begin{thebibliography}{99}

\bibitem{AransonKramer2002}
I.~S. Aranson and L.~Kramer, \emph{The world of the complex {Ginzburg--Landau}
  equation}, Rev. Modern Phys. \textbf{74} (2002), no.~1, 99--143,
  \url{https://doi.org/10.1103/RevModPhys.74.99}.

\bibitem{Aubin1976}
T.~Aubin, \emph{Probl\`emes isop\'erim\'etriques et espaces de {S}obolev}, J.
  Diff. Geo. \textbf{11} (1976), no.~4, 573--598,
  \url{https://doi.org/10.4310/jdg/1214433725}.

\bibitem{BechoucheJungel2000}
P.~Bechouche and A.~J\"ungel, \emph{Inviscid limits of the complex
  {Ginzburg--Landau} equation}, Comm. Math. Phys. \textbf{214} (2000), no.~1,
  201--226, \url{https://doi.org/10.1007/s002200000263}.

\bibitem{CazenaveWeissler1990}
T.~Cazenave and F.~B. Weissler, \emph{The {C}auchy problem for the critical
  nonlinear {S}chr\"odinger equation in {$H^s$}}, Nonlinear Anal. \textbf{14}
  (1990), no.~10, 807--836,
  \url{https://doi.org/10.1016/0362-546X(90)90023-A}.

\bibitem{ChengGuoZhengZheng2024Exterior}
X.~Cheng, C.-Y. Guo, J.~Zheng, and Y.~Zheng, \emph{Global well-posedness of the
  energy-critical complex {Ginzburg--Landau} equation in exterior domains},
  Calc. Var. Partial Differential Equations \textbf{64} (2025), no.~9, Paper
  No. 298, 55, \url{https://doi.org/10.1007/s00526-025-03165-5}.

\bibitem{ChengGuoZheng2026JMPA}
X.~Cheng, C.-Y. Guo, and Y.~Zheng, \emph{Global weak solution of 3-{D} focusing
  energy-critical nonlinear {S}chr\"odinger equation}, J. Math. Pures Appl. (9)
  \textbf{214} (2026), Paper No. 103947,
  \url{https://doi.org/10.1016/j.matpur.2026.103947}.

\bibitem{ChengGuoZheng2025Limit}
X.~Cheng, C.~Guo, and Y.~Zheng, \emph{The limit theory of energy-critical
  complex {Ginzburg--Landau} equation}, Acta Math. Sin. (Engl. Ser.)
  \textbf{41} (2025), no.~12, 3003--3019,
  \url{https://doi.org/10.1007/s10114-025-4310-3}.

\bibitem{ChristKiselev2001}
M.~Christ and A.~Kiselev, \emph{Maximal functions associated to filtrations},
  J. Funct. Anal. \textbf{179} (2001), no.~2, 409--425,
  \url{https://doi.org/10.1006/jfan.2000.3687}.

\bibitem{ClementOkazawaSobajimaYokota2012}
P.~Cl{\'e}ment, N.~Okazawa, M.~Sobajima, and T.~Yokota, \emph{A simple approach
  to the {C}auchy problem for complex {Ginzburg--Landau} equations by
  compactness methods}, J. Diff. Equ. \textbf{253} (2012), no.~4,
  1250--1263, \url{https://doi.org/10.1016/j.jde.2012.05.002}.

\bibitem{CorreiaFigueira2018}
S.~Correia and M.~Figueira, \emph{Some stability results for the complex
  {Ginzburg--Landau} equation}, Commun. Contemp. Math. \textbf{22} (2020),
  no.~8, Paper No. 1950038,
  \url{https://doi.org/10.1142/S021919971950038X}.

\bibitem{Giga1986}
Y.~Giga, \emph{Solutions for semilinear parabolic equations in {$L^p$} and
  regularity of weak solutions of the {N}avier--{S}tokes system}, J.
  Diff. Equ. \textbf{62} (1986), no.~2, 186--212,
  \url{https://doi.org/10.1016/0022-0396(86)90096-3}.

\bibitem{GinzburgLandau1950}
V.~L. Ginzburg and L.~D. Landau, \emph{On the theory of superconductivity},
  Zh. Eksperim. Teor. Fiz. \textbf{20} (1950), no.~12, 1064--1082.

\bibitem{GinibreVelo1992}
J.~Ginibre and G.~Velo, \emph{Smoothing properties and retarded estimates for
  some dispersive evolution equations}, Comm. Math. Phys. \textbf{144} (1992),
  no.~1, 163--188, \url{https://doi.org/10.1007/BF02099195}.

\bibitem{GinibreVelo1996}
J.~Ginibre and G.~Velo, \emph{The {C}auchy problem in local spaces for the
  complex {Ginzburg--Landau} equation. {I}. {Compactness} methods}, Phys. D
  \textbf{95} (1996), no.~3--4, 191--228,
  \url{https://doi.org/10.1016/0167-2789(96)00055-3}.

\bibitem{GinibreVelo1997}
J.~Ginibre and G.~Velo, \emph{The {C}auchy problem in local spaces for the
  complex {Ginzburg--Landau} equation. {II}. {Contraction} methods}, Comm.
  Math. Phys. \textbf{187} (1997), no.~1, 45--79,
  \url{https://doi.org/10.1007/s002200050129}.

\bibitem{GuoJiangLi2020}
B.~Guo, M.~Jiang, and Y.~Li, \emph{Ginzburg--Landau equations}, Science Press,
  Beijing, 2020.

\bibitem{Henry1981}
D.~Henry, \emph{Geometric theory of semilinear parabolic equations}, Lecture
  Notes in Mathematics \textbf{840}, Springer-Verlag, Berlin--New York, 1981,
  \url{https://doi.org/10.1007/BFb0089647}.

\bibitem{HuangWang2008}
C.~Huang and B.~Wang, \emph{Inviscid limit for the energy-critical complex
  {Ginzburg--Landau} equation}, J. Funct. Anal. \textbf{255} (2008), no.~3,
  681--725, \url{https://doi.org/10.1016/j.jfa.2008.04.017}.

\bibitem{KeelTao1998}
M.~Keel and T.~Tao, \emph{Endpoint {S}trichartz estimates}, Amer. J. Math.
  \textbf{120} (1998), no.~5, 955--980,
  \url{https://doi.org/10.1353/ajm.1998.0039}.

\bibitem{Lunardi1995}
A.~Lunardi, \emph{Analytic semigroups and optimal regularity in parabolic
  problems}, Progress in Nonlinear Differential Equations and Their
  Applications \textbf{16}, Birkh{\"a}user Verlag, Basel, 1995,
  \url{https://doi.org/10.1007/978-3-0348-9234-6}.

\bibitem{MachiharaNakamura2003}
S.~Machihara and Y.~Nakamura, \emph{The inviscid limit for the complex
  {Ginzburg--Landau} equation}, J. Math. Anal. Appl. \textbf{281} (2003),
  no.~2, 552--564,
  \url{https://doi.org/10.1016/S0022-247X(03)00143-4}.

\bibitem{MiaoYuanZhang2008}
C.~Miao, B.~Yuan, and B.~Zhang, \emph{Well-posedness of the {C}auchy problem
  for the fractional power dissipative equations}, Nonlinear Anal. \textbf{68}
  (2008), no.~3, 461--484,
  \url{https://doi.org/10.1016/j.na.2006.11.011}.

\bibitem{MiaoZhangZheng2025}
C.~Miao, B.~Zhang, and J.~Zheng, \emph{Harmonic analysis methods in partial
  differential equations}, De Gruyter Studies in Mathematics \textbf{102}, De
  Gruyter, Berlin--Boston, 2025,
  \url{https://doi.org/10.1515/9783111384726}.

\bibitem{OkazawaYokota2002JDE}
N.~Okazawa and T.~Yokota, \emph{Global existence and smoothing effect for the
  complex {Ginzburg--Landau} equation with {$p$}-{L}aplacian}, J. Diff. Equ.
  \textbf{182} (2002), no.~2, 541--576,
  \url{https://doi.org/10.1006/jdeq.2001.4097}.

\bibitem{OkazawaYokota2002Monotonicity}
N.~Okazawa and T.~Yokota, \emph{Monotonicity method applied to the complex
  {Ginzburg--Landau} and related equations}, J. Math. Anal. Appl. \textbf{267}
  (2002), no.~1, 247--263,
  \url{https://doi.org/10.1006/jmaa.2001.7770}.

\bibitem{Pazy1983}
A.~Pazy, \emph{Semigroups of linear operators and applications to partial
  differential equations}, Applied Mathematical Sciences \textbf{44},
  Springer-Verlag, New York, 1983,
  \url{https://doi.org/10.1007/978-1-4612-5561-1}.

\bibitem{Stein1970}
E.~M. Stein, \emph{Singular integrals and differentiability properties of
  functions}, Princeton Mathematical Series \textbf{30}, Princeton University
  Press, Princeton, NJ, 1970,
  \url{https://doi.org/10.1515/9781400883882}.

\bibitem{Strichartz1977}
R.~S. Strichartz, \emph{Restrictions of {F}ourier transforms to quadratic
  surfaces and decay of solutions of wave equations}, Duke Math. J. \textbf{44}
  (1977), no.~3, 705--714,
  \url{https://doi.org/10.1215/S0012-7094-77-04430-1}.

\bibitem{Talenti1976}
G.~Talenti, \emph{Best constant in {S}obolev inequality}, Ann. Mat. Pura Appl.
  (4) \textbf{110} (1976), no.~1, 353--372,
  \url{https://doi.org/10.1007/BF02418013}.

\bibitem{Tao2006}
T.~Tao, \emph{Nonlinear dispersive equations: Local and global analysis}, CBMS
  Regional Conference Series in Mathematics \textbf{106}, American Mathematical
  Society, Providence, RI, 2006,
  \url{https://doi.org/10.1090/cbms/106}.

\bibitem{Wang2002}
B.~Wang, \emph{The limit behavior of solutions for the {C}auchy problem of
  the complex {Ginzburg--Landau} equation}, Comm. Pure Appl. Math. \textbf{55}
  (2002), no.~4, 481--508, \url{https://doi.org/10.1002/cpa.10024}.

\bibitem{Weissler1980}
F.~B. Weissler, \emph{Local existence and nonexistence for semilinear parabolic
  equations in {$L^p$}}, Indiana Univ. Math. J. \textbf{29} (1980), no.~1,
  79--102, \url{https://doi.org/10.1512/iumj.1980.29.29007}.

\bibitem{Wu1998}
J.~Wu, \emph{The inviscid limit of the complex {Ginzburg--Landau} equation},
  J. Diff. Equ. \textbf{142} (1998), no.~2, 413--433,
  \url{https://doi.org/10.1006/jdeq.1997.3347}.

\end{thebibliography}

\end{document}